\documentclass{amsart}
\usepackage[T1]{fontenc}
\usepackage{amssymb}
\usepackage{amsmath}
\usepackage{amsthm}
\usepackage{hyperref}
\usepackage{comment}
\usepackage[noabbrev]{cleveref} 
\usepackage{color}
\usepackage{enumerate}

\makeatletter
\newcommand{\addresshere}{%
  \enddoc@text\let\enddoc@text\relax
}
\makeatother

\definecolor{blue}{rgb}{0,0,1}
\definecolor{green}{rgb}{0, 0.5, 0}

\newcommand\Z{\mathbb{Z}}
\newcommand\Q{\mathbb{Q}}
\renewcommand\P{\mathbb{P}}
\renewcommand\L{\mathbb{L}}
\DeclareMathOperator{\diag}{diag}
\DeclareMathOperator{\SL}{SL}

\DeclareMathOperator{\lcm}{lcm}
\DeclareMathOperator{\ord}{ord}
\DeclareMathOperator{\wt}{wt}

\DeclareMathOperator{\prim}{prim}

\newcommand{\A}{\mathcal{A}}
\newcommand{\B}{\mathcal{B}}
\newcommand{\C}{\mathcal{C}}

\newcommand{\software}[1]{\textsc{#1}{}}
\newcommand{\Sage}{\software{SageMath}}

\theoremstyle{plain}
\newtheorem{theorem}{Theorem}[section]
\newtheorem{corollary}[theorem]{Corollary}
\newtheorem{lemma}[theorem]{Lemma}
\newtheorem{proposition}[theorem]{Proposition}
\newtheorem*{problem}{Main Problem}

\newtheorem{theoremvar}{Theorem}
\newenvironment{theorem*}[1]
 {\renewcommand{\thetheoremvar}{#1}\theoremvar}
 {\endtheoremvar}

\theoremstyle{definition}
\newtheorem{definition}[theorem]{Definition}
\newtheorem{example}{Example}
\newtheorem*{example*}{Example}

\theoremstyle{remark}
\newtheorem{remark}[theorem]{Remark}

\begin{document}

\title{Lattice coverings and homogeneous covering congruences}
\author{J. E. Cremona}
\address{Mathematics Institute, University of Warwick, Coventry, CV4
  7AL, UK.}
\email{J.E.Cremona@warwick.ac.uk}
\author{P. Koymans}
\address{Mathematisch Instituut, Universiteit Utrecht, Postbus 80.010, 3508 TA Utrecht, The Netherlands.}
\email{p.h.koymans@uu.nl}

\begin{abstract}
We consider the problem of covering $\Z^2$ with a finite number of
sublattices of finite index; this may be viewed as a projective, or
homogeneous, version of the well-known problem of covering systems of
congruences. We give a construction of minimal coverings which
produces many, but not all, minimal coverings, prove an inequality
relating the size of a covering to the least common multiple of the
lattice indices, and determine all minimal coverings with at most~$8$
sublattices.
\end{abstract}

\maketitle



\section{Introduction}
We consider the problem of covering $\Z^2$ with a finite number of
sublattices of finite index; this may be viewed as a projective (or
homogeneous) version of the well-known problem of covering~$\Z$ with a
``covering system'' of residue classes.  The focus of this paper is to
understand the structure of such homogeneous covering systems, and
explore systematic constructions of them.

We will use the term \emph{lattice} exclusively for a subgroup of
$\Z^2$ of finite index, and call a lattice $L$ \emph{cocyclic} if the
quotient $\Z^2/L$ is a (finite) \emph{cyclic} group: see Definition
\ref{def:lattice}, and \Cref{sec:lattices} for basic facts about
these.

In order to state our main results, we first define what we mean by a
\emph{lattice covering}, and introduce three natural classes of these,
which become progressively more restrictive: see~\Cref{sec:covering}
for more details.

\begin{definition}
\label{def:covering}
  A finite collection of lattices $\C=\{L_1, \dots,L_n\}$ is said to
\emph{cover}~$\Z^2$, or to be a \emph{covering} of~$\Z^2$, if
\begin{equation}\label{eqn:covering}
   \Z^2 = L_1 \cup L_2 \cup \dots \cup L_n.
\end{equation}
A covering~$\C$ is~\emph{irredundant} if no proper subset of~$\C$ is a
covering; that is, no lattice in the covering is contained in the
union of the others.

A covering is \emph{minimal} if (\ref{eqn:covering}) fails on
replacing any~$L_i$ by a proper sublattice.  In particular, a minimal
covering is irredundant.

A minimal covering~$\{L_i\}$ is said to be \emph{strongly minimal} if every
primitive vector of~$\Z^2$ is contained in \emph{exactly one} of
the~$L_i$.
\end{definition}

One important construction in our paper is that of \emph{refinement},
which constructs a new covering from a given covering by replacing a
lattice in the covering with certain cocyclic sublattices.  This
process is described in~\Cref{sec:refinement};
in particular, a \emph{$p$-refinement}
(see~Definition~\ref{def:p-refinement}) replaces one lattice in the
covering by all of its cocyclic sublattices of
relative index~$p$, for some prime~$p$.
We show that refinement sends
strongly minimal coverings to strongly minimal coverings, but give
examples to show that not all strongly minimal coverings are obtained
(from the trivial covering $\{\Z^2\}$)
by this construction, and that not all minimal coverings are strongly
minimal: see respectively Examples~\ref{ex:strong-not-refinement}
and~\ref{ex:min-not-strong} below.

We show that, for coverings by lattices all of whose indices are
powers of the same prime~$p$, the three notions of minimality
coincide, and in~\Cref{sec:refinement} we show that all such coverings
are obtained by repeated $p$-refinement:

\begin{theorem*}{A}[{\Cref{thm:all-p-power}}]
\label{thm:A}
  Let~$\C$ be an irredundant covering.  If all the
  indices of the lattices in~$\C$ are powers of one prime~$p$,
  then~$\C$ is obtained from the trivial covering~$\{\Z^2\}$ by
  successive $p$-refinement operations.
\end{theorem*}

In ~\Cref{sec:lower-bounds} we prove a result relating the size of a
minimal covering system to the indices of the lattices in the system:

\begin{theorem*}{B}[{\Cref{thm:Simpson-theorem-2-lattice}}]
  \label{thm:B}
  Let~$\C$ be an irredundant covering.  If the least
  common multiple of the indices of the lattices in~$\C$ has prime
  factorization~$\prod_{i=1}^{t}p_i^{e_i}$, then the number of
  lattices in~$\C$ is at least $t+1+\sum_{i=1}^{t}e_i(p_i-1)$.
\end{theorem*}

\Cref{thm:B}, and its proof, were inspired by a similar result by
Simpson~\cite{Simpson} for systems of covering congruences, where the
corresponding lower bound is given by~$1+\sum_{i=1}^{t}e_i(p_i-1)$.
It is best possible, the lower bound being attained by certain
coverings which are constructed by~refinement:
see~\Cref{thm:refinements-lower-bound}.  Our result allows us to
easily recover Simpson's original result.

In~\Cref{sec:small-size} we list all minimal coverings by at most~$8$
sublattices, showing that those of size at most~$6$ are all strongly
minimal, but that some minimal coverings of size~$7$ and~$8$ are not
strongly minimal.  We summarise this classification here, referring
to~\Cref{sec:small-size} for details.

\begin{theorem*}{C}[{\Cref{thm:all-upto-6}, \Cref{thm:all-7}, \Cref{thm:all-8}}]
  \label{thm:C}
  There are~$55$ minimal lattice coverings of size at most~$6$, all of
  which are strongly minimal and obtained by refining the trivial
  covering.  They are listed in~\Cref{tab:size-upto-6} in~\Cref{sec:small-size}.

There are~$144$ minimal coverings of size~$7$, including~$126$ which
are strongly minimal and~$18$ which are not strongly minimal.  They
are listed in~\Cref{tab:size-7} in~\Cref{sec:small-size}.

There are~$724$ minimal coverings of size~$8$, including~$550$ which
are strongly minimal and~$174$ which are not strongly minimal.  They
are listed in~\Cref{tab:size-8} in Appendix~A.
\end{theorem*}

A key observation in our paper, explored
in~\Cref{sec:projective_coverings}, is that the lattice covering
problem is equivalent to a projective version of the classical
congruence covering problem (see~\Cref{sec:congruences} for a detailed
comparison between lattice coverings and the congruence covering
problem).  The congruence covering problem was introduced by
Erd\H{o}s.  The first use of covering congruence systems, in 1950, was
to show that a positive proportion of integers are not of the
shape $2^k + p$, see~\cite{Erdos1}. Erd\H{o}s~\cite{Erdos2} formally
introduced them in 1952, and returned many times to the topic, raising
many influential questions and conjectures. This led to a vast amount
of literature on the topic, including~\cite{Hough, HN}
and~\cite{BBMST}.  We refer to the work of Balister~\cite{Balister}
for an excellent overview of the topic.

We introduce the projective (or homogeneous) covering problem
in~\Cref{sec:projective_coverings}, showing how it leads naturally to
a restatement in terms of cocyclic lattices, whose definition and
basic properties are given in~\Cref{sec:lattices}.  There, we show
that cocyclic lattices of index~$N$ in~$\Z^2$ are parametrised by the
projective line~$\P^1(\Z/N\Z)$
(see~Corollary~\ref{cor:P1N-LcdN-bijection}).  This allows us to
reformulate the projective covering problem in terms of coverings
of~$\Z^2$ by lattices, in~\Cref{sec:covering}.  The concept of
refinement is introduced in~\Cref{sec:refinement}, where we
prove~\Cref{thm:A} (see \Cref{thm:all-p-power}) as well as a lower
bound on the size of a refinement of the trivial covering in terms of
the least common multiple of the indices of the lattices in the
covering (\Cref{thm:refinements-lower-bound}).  This lower bound
result is extended to apply to all irredundant coverings
in~\Cref{sec:lower-bounds} where we prove~\Cref{thm:B} (see
\Cref{thm:Simpson-theorem-2-lattice}). In \Cref{sec:small-size} we
apply the preceding results to classify and describe all minimal
lattice coverings of size at most~$8$, as in~\Cref{thm:C}.  Finally,
\Cref{sec:congruences} summarises the analogues of our results on
lattice coverings in the context of covering congruences.


\subsection{Acknowledgements}
This study arose out of a series of papers by the second author with
\'E.~Fouvry, which study the value sets of integral binary forms.  A key
ingredient in those papers is an enumeration of all minimal lattice
coverings (in the sense we define here) of size at most~$6$.  This
list can be easily reconstructed using our results
(see~\Cref{sec:small-size}).  It was after the first author attended
an online seminar given by the second author on his work with Fouvry
that this study began, and both authors would like to thank Professor
G.~Kálmán and the organisers of the Debrecen online Number Theory
Seminar for stimulating their collaboration.  The first author gratefully acknowledges the support of the Dutch Research Council (NWO) through the Veni grant ``New methods in arithmetic statistics''.


\section{Covering systems and projective covering systems}
\label{sec:projective_coverings}
Recall that a \emph{covering system of congruences} is a finite system
of congruences $x\equiv a_i\pmod{N_i}$ for $1\le i\le n$ such that
every integer satisfies at least one of the congruences.
Writing~$R(a;N)$ for the set of integers congruent to~$a\pmod{N}$, the
covering condition is simply that
\[
\Z=\bigcup_{i=1}^{n}R(a_i;N_i).
\]
The simplest such systems consist, for any positive integer~$N$, of
all the residue classes~$R(a;N)$ for~$0\le a<N$.

For each~$N$ we have a surjective reduction map~$\Z\to\Z/N\Z$,
and~$R(a;N)$ is the preimage of one element of~$\Z/N\Z$, namely the
class of~$a\pmod{N}$.  From this viewpoint, a covering system of
congruences consists of a finite set of elements in~$\Z/N_i\Z$ (for
various moduli~$N_i$) whose preimages in~$\Z$ cover~$\Z$.

We now consider a projective version of this.  Recall that for any
ring~$R$ (commutative and with~$1$), the~\emph{projective
line}~$\P(R)$ is defined as
\[
\P(R) = \{(c,d)\in R^2\mid cR+dR=R\} / {}\sim{},
\]
where the equivalence relation~${}\sim{}$ is defined by
\[
(c_1,d_1)\sim(c_2,d_2) \iff c_1d_2=c_2d_1.
\]
For $v_i=(c_i,d_i)\in R^2$, we write~$v_1\wedge v_2=c_1d_2-c_2d_1$;
then the equivalence relation may be written
\[
v_1\sim v_2 \iff v_1\wedge v_2 = 0.
\]
Note that the coprimality condition~$cR+dR=R$ is that~$a,b\in R$ exist
with~$ad-bc=1$, or equivalently that $(c,d)$ is the second row of a
matrix in~$\SL(2,R)$. Note also that
\[
  v_1\sim v_2\iff\exists u\in R^{\times}:\ uc_1=c_2, ud_1=d_2.
\]
The backward implication is trivial. For the forward implication, one may check that if $a_id_i-b_ic_i=1$ for~$i=1,2$ and
$c_1d_2=d_1c_2$, then~$u=a_1d_2-b_1c_2$ is a unit with
inverse~$a_2d_1-b_2c_1$. (See \cite[Proposition~2.2.1]{JCbook},
where~$R=\Z$, but the proof there works in general.)

We denote the class of~$(c,d)$ in~$\P(R)$ by~$(c:d)$.  For~$R=\Z$,
elements~$(c:d)\in\P(\Z)$ are represented by \emph{primitive}
vectors~$(c,d)\in\Z^2$ (a vector $(c,d)\in\Z^2$ is primitive
if~$\gcd(c,d)=1$), uniquely up to sign.

If~$\phi:R\to S$ is a ring homomorphism, there is an induced
well-defined map from~$\P(R)$ to~$\P(S)$, sending~$(c:d)$
to~$(\phi(c):\phi(d))$.  In particular, we have reduction
maps~$\P(\Z)\to\P(\Z/N\Z)$ for all positive integers~$N$. As
these~$\P(\Z/N\Z)$ feature frequently in what follows, we will
denote~$\P(\Z/N\Z)$ simply by~$\P(N)$.  Let~$\psi(N)$ denote the
cardinality of~$\P(N)$; a formula for~$\psi(N)$ is given
in~\Cref{cor:psi-formula} below.

Note that for two integers~$c,d$ to define an element of~$\P(N)$ via
reduction modulo~$N$, we only require that~$\gcd(c,d,N)=1$, not
that~$c,d$ are coprime.  However there always exists a representative
pair which are actually coprime:
\begin{lemma}
  \label{lem:cd-coprime}
  For all integers~$N\ge1$, the reduction map $\P(\Z)\to\P(N)$ is
  surjective.
\end{lemma}
\begin{proof}
  This follows from the fact that the reduction
  map~$\SL(2,\Z)\to\SL(2,\Z/N\Z)$ is surjective. See Lemma 1.38 of
  \cite{Shimura} for one proof of this.
\end{proof}
Given two integers~$c,d$ with~$\gcd(c,d,N)=1$, let~$(c:d)_N$ denote
the element of~$\P(N)$ defined by the reductions of~$c$ and~$d$
modulo~$N$. While every element of~$\P(N)$ may be written as~$(c:d)_N$
with $\gcd(c,d)=1$ by~\Cref{lem:cd-coprime}, when we write such an
element as~$(c:d)_N$ we only assume the weaker condition, unless
explicitly stated. Then by definition,
\begin{align*}
(c_1:d_1)_N=(c_2:d_2)_N &\iff
c_1d_2\equiv c_2d_1\pmod{N} \\ &\iff
(c_1:d_1)\wedge(c_2:d_2) \equiv0\pmod{N}.
\end{align*}
We denote the preimage of~$(c:d)_N$
in~$\P(\Z)$ by
\begin{align*}
  R(c:d;N) &= \{(x:y)\in\P(\Z)\mid dx\equiv cy\pmod{N}\}\\
  &= \{(x:y)\in\P(\Z)\mid (x,y)\wedge(c,d)\equiv 0\pmod{N}\}.
\end{align*}
We call the~$\psi(N)$ subsets of~$\P(\Z)$ defined this way the~\emph{homogeneous residue classes} modulo~$N$
in~$\P(\Z)$.

\begin{remark}
  The sets~$\P(\Z)$ and~$\P(\Q)$ may be identified, since it is
  easy to see that the map~$\P(\Z)\to\P(\Q)$, induced by the
  inclusion~$\Z\to\Q$, is a bijection.  (Similarly for any principal
  ideal domain and its field of fractions.)  For our purposes, it is
  more convenient to use~$\P(\Z)$, as its elements have an almost
  unique representation as a pair of coprime integers.
\end{remark}

\begin{lemma}
  \label{lem:reduction-P1N-P1M}
  Let~$M,N$ be positive integers with $M\mid N$.  The map
  $\P(N)\to\P(M)$ defined by~$(c:d)_N\mapsto(c:d)_M$ is well-defined,
  and surjective.
\end{lemma}
\begin{proof}
  It is clear that the map is well-defined, noting
  that~$\gcd(c,d,N)=1$ implies $\gcd(c,d,M)=1$ and that~$(\Z/N\Z)^\times$ maps to $(\Z/M\Z)^\times$. For surjectivity, if we
  write an element of~$\P(M)$ in the special form~$(c:d)_M$ with $\gcd(c,d)=1$
  (using~\Cref{lem:cd-coprime}), it is then the image
  of the element~$(c:d)_N\in\P(N)$.
\end{proof}

The following is a simple projective version of the Chinese Remainder
Theorem. See~\Cref{prop:CRT-P1-full} below for a more general
statement (and proof), where the moduli are not necessarily coprime.
\begin{proposition}\label{prop:CRT-P1}
  Let~$N_1$, $N_2$ be coprime integers, and set $N=N_1N_2$.  The
  surjective maps~$\P(N) \to \P(N_i)$ for~$i=1,2$
  of~\Cref{lem:reduction-P1N-P1M} induce a map~$\P(N) \to
  \P(N_1)\times\P(N_2)$ which is a bijection.  Hence~$\psi$ is a
  multiplicative function: $\psi(N)=\psi(N_1)\psi(N_2)$
  when~$\gcd(N_1,N_2)=1$.
\end{proposition}

\begin{lemma}
  \label{lem:index-p}
  Let~$p$ be prime and~$M\ge1$. Under the surjective
  map~$\P(pM)\to\P(M)$, each element of~$\P(M)$ has~$p$
  preimages in~$\P(pM)$ if~$p\mid M$ and~$p+1$ preimages in~$\P(pM)$ if~$p\nmid M$.  Hence
  \begin{equation}\label{eqn:psi-p-formula}
  \psi(pM) = \begin{cases} p\psi(M) & \text{if $p\mid M$;}\\
    (p+1)\psi(M) & \text{if $p\nmid M$.}
  \end{cases}
  \end{equation}
\end{lemma}
\begin{proof}
  If~$p\nmid M$ then the map~$\P(pM)\to\P(M)$ is the composite of the
  bijection $\P(pM) \to \P(p)\times\P(M)$ of~\Cref{prop:CRT-P1} with
  projection onto the second factor.  Each fibre of the map has
  cardinality~$\psi(p)=\#\P(p)=p+1$.  Explicitly, let~$c_0,d_0$ be
  coprime integers representing an element~$(c_0:d_0)_M\in\P(M)$, and
  let~$c',d'$ run through coprime pairs representing the~$p+1$
  elements of~$\P(p)$.  Then the $p+1$ lifts of~$(c_0:d_0)_M$
  to~$\P(Mp)$ are obtained by applying the Chinese Remainder Theorem
  to give integers~$c,d$ with~$(c,d)\equiv(c_0,d_0)\pmod{M}$
  and~$(c,d)\equiv(c',d')\pmod{p}$; while~$c$ and~$d$ are not
  necessarily coprime, we do have~$\gcd(c,d,pM)=1$, so that~$(c:d)_{pM}$
  is well-defined.

  In this case we have $\psi(pM)=\psi(p)\psi(M)=(p+1)\psi(M)$.

  Now suppose that $p\mid M$.  We claim that each $(c:d)_M$ has
  exactly~$p$ lifts to~$\P(pM)$, implying that~$\psi(pM)=p\psi(M)$ in
  this case.  By the Chinese Remainder Theorem again, it suffices to
  assume that~$M$ is a power of~$p$, say~$M=p^k$ with~$k\ge1$.  Now
  if~$p\nmid d$ we may assume that~$d=1$, and the lifts
  of~$(c:1)_{p^k}$ are~$(c+tp^k:1)_{p^{k+1}}$ for $0\le
  t<p$. Similarly if~$p\nmid c$.
\end{proof}


\begin{corollary}
  \label{cor:psi-formula}
  The function $\psi(N) = \#\P(N)$ is given by
  \begin{equation}\label{eqn:psi-formula}
    \psi(N) = N \prod_{p \mid N} (1+1/p),
  \end{equation}
  the product being over all primes dividing~$N$.
\end{corollary}
\begin{proof}
  This formula (which is standard, see Proposition 1.43 (1) in
  \cite{Shimura}), follows by induction from \cref{eqn:psi-p-formula}.
\end{proof}

Note that for each residue class~$R(a;N)$ and divisor~$M\mid N$, there
is a unique residue class modulo~$M$ containing~$R(a;N)$,
namely~$R(a;M)$.  Then one way to state the usual Chinese Remainder
Theorem is as follows.  Given two residue classes~$R(a_1,N_1)$
and~$R(a_2;N_2)$, the following are equivalent:
\begin{enumerate}[(i)]
\item $R(a_1,N_1)$ and~$R(a_2;N_2)$ are contained in the same residue
  class modulo~$M=\gcd(N_1,N_2)$;
\item $R(a_1,N_1)\cap R(a_2;N_2)$ is a residue class
  modulo~$N=\lcm(N_1,N_2)$.
\end{enumerate}
In case conditions~(i) and~(ii) fail, the intersection in~(ii) is empty.

The projective version of this will be useful in the context of
lattices.
\begin{proposition}
\label{prop:CRT-P1-full}
  Given~$(c_1:d_1)_{N_1}\in\P(N_1)$ and~$(c_2:d_2)_{N_2}\in\P(N_2)$, the
following are equivalent, where~$M=\gcd(N_1,N_2)$ and~$N=\lcm(N_1,N_2)$:
\begin{enumerate}[(i)]
\item $(c_1:d_1)_{N_1}$ and~$(c_2:d_2)_{N_2}$ have the same image
  in~$\P(M)$;
\item $(c_1:d_1)_{N_1}$ and~$(c_2:d_2)_{N_2}$ have a common preimage
  in~$\P(N)$.
\end{enumerate}
In terms of homogeneous residue classes, $R(c_1:d_1;N_1)$
and~$R(c_2:d_2;N_2)$ have nonempty intersection if and only if they
are contained in the same homogeneous residue class modulo~$M$, or
equivalently, $(c_1,d_1)\wedge(c_2,d_2)\equiv0\pmod{M}$;
in this case, the intersection is a homogeneous residue class
modulo~$N$.
\end{proposition}
\begin{proof}
  (ii)$\implies$(i): Given~(ii), there exist coprime integers~$c,d$
  defining~$(c:d)_N\in\P(N)$ such that $(c:d)_{N_1}=(c_1:d_1)_{N_1}$
  and $(c:d)_{N_2}=(c_2:d_2)_{N_2}$.  But then (by the well-defined
  property of reduction in~\Cref{lem:reduction-P1N-P1M}) we have
  $(c_1:d_1)_M = (c:d)_M = (c_2:d_2)_M$.

  (i)$\implies$(ii): Given~(i), $(c_1:d_1)_M=(c_2:d_2)_M$, so there
  exists~$u$ coprime to~$M$ such that $uc_1\equiv c_2$, $ud_1\equiv
  d_2\pmod{M}$. We may take~$u$ to be coprime to~$N_1$, since the map
  $(\Z/N_1\Z)^\times\to(\Z/M\Z)^\times$ is surjective. Then
  $(uc_1:ud_1)_{N_1}=(c_1:d_1)_{N_1}$.  By the usual Chinese Remainder
  Theorem, there exist integers~$c,d$ (unique modulo~$N$) congruent
  to~$uc_1,ud_1$ modulo~$N_1$ and to~$c_2,d_2$ modulo~$N_2$.  We
  have~$\gcd(c,d,N)=1$, since $\gcd(c,d,N_1)=\gcd(uc_1,ud_1,N_1)=1$
  and~$\gcd(c,d,N_2)=\gcd(c_2,d_2,N_2)=1$. Hence $(c:d)_N$ is a valid
  element of~$\P(N)$, and it is a common preimage of~$(c_1:d_1)_{N_1}$
  and~$(c_2:d_2)_{N_2}$.
\end{proof}

We now state the projective version of the covering congruence
problem.
\begin{problem} Describe all finite sets~$\{(c_i:d_i)_{N_i} \in\P(N_i) \mid
    1\le i\le n\}$ such that
    \[
    \P(\Z) = \bigcup_{i=1}^{n}R(c_i:d_i;N_i).
    \]
\end{problem}
In other words, we seek to cover~$\P(\Z)$ with a finite set of
homogeneous residue classes.  As with the standard covering congruence
problem, there are simple solutions to the problem, for each~$N$,
with~$n=\psi(N)$ and $N_i=N$ for all~$i$, consisting of all~$\psi(N)$
elements of~$\P(N)$.  The trivial covering arises when $N=1$,
since~$R(0:1;1)=\P(\Z)$.

In the next section we will translate the projective congruence
covering problem into a different form, where we seek to cover~$\Z^2$
with a finite set of~\emph{lattices}: subgroups of finite index
in~$\Z^2$.  These may be constructed as follows.  Starting
with~$(c:d)_N\in\P(N)$ (where~$\gcd(c,d,N)=1$), lift the associated
residue class~$R(c:d;N)\subseteq\P(\Z)$ to~$\Z^2$, to obtain the set
of primitive vectors~$(x,y)\in\Z^2$ satisfying~$cy\equiv dx\pmod{N}$.
Now, dropping the coprimality condition on~$x,y$, define
\begin{equation}
\label{eqn:LcdN-def}
L(c:d;N) = \{(x,y)\in\Z^2\mid cy\equiv dx\pmod{N}\}.
\end{equation}
If we set~$v=(c,d)\in\Z^2$, the same set will be denoted~$L(v;N)$, so
\[
L(v;N) = \{w\in\Z^2\mid v\wedge w\equiv0\pmod{N}\}.
\]
This is a lattice in~$\Z^2$ (as defined in~\Cref{def:lattice} below),
of index~$N$. One basis of ~$L(v;N)$ is~$\{N(a,b),(c,d)\}$, where
$ad-bc=1$; it follows that the quotient~$\Z^2/L(v;N)$ is therefore
cyclic of order~$N$.  The primitive vectors in~$L(v;N)$ are precisely
the lifts to~$\Z^2$ of the preimage in~$\P(\Z)$ of~$R(c:d;N)$.


\section{Cocyclic lattices}
\label{sec:lattices}
In this section we define and study lattices and cocyclic lattices in
$\Z^2$, showing in~\Cref{cor:P1N-LcdN-bijection} that for
each~$N\ge1$, the map~$(c:d)_N\mapsto L(c:d;N)$ is a bijection
between~$\P(N)$ and the set~$\L(N)$ of cocyclic lattices of index~$N$.

\begin{definition}
  \label{def:lattice}
  In this paper, a \emph{lattice} is a subgroup of $\Z^2$ of finite
  index. A lattice $L$ is \emph{cocyclic} if the finite quotient group
  $\Z^2/L$ is cyclic; equivalently, $L$ is not contained in $c\Z^2$
  for any integer~$c>1$.

  The set of all cocyclic lattices of index~$N$ is denoted~$\L(N)$.
\end{definition}

We write elements of~$\Z^2$ as row vectors.  Every lattice~$L$ of
index~$N$ has the form~$L=\Z^2A$, where~$A$ is a~$2\times2$ integer
matrix with determinant~$\det(A)=N$, the rows of~$A$ forming
a~$\Z$-basis of~$L$.  For~$L$ to be cocyclic, a necessary and
sufficient condition is that~$A$ is \emph{primitive} (that is, having
coprime entries), or equivalently that the Smith normal form of~$A$
has the form~$D_N=\diag(N,1)$.  Then~$A=UD_NV$ with
$U,V\in\Gamma:=\SL(2,\Z)$, and~$L=\Z^2UD_NV=\Z^2D_NV$.  Hence
\[
\L(N) = \{\Z^2D_NV \mid V\in\Gamma\};
\]
however, different matrices~$V\in\Gamma$ can result in the same
lattice~$\Z^2D_NV$, as we now make precise.

Recall that the congruence subgroup $\Gamma_0(N)$ of~$\Gamma$ is
defined as
\[
\Gamma_0(N) = \Gamma \cap D_N^{-1}\Gamma D_N =
\left\{\begin{pmatrix}a&b\\ c&d\end{pmatrix}\in\Gamma\mid
c\equiv0\pmod{N}\right\}.
\]
\begin{lemma}
  \label{lem:Gamma0N_cosets}
  Let~$N$ be a positive integer, $D_N=\diag(N,1)$, and
  $V_i=\begin{pmatrix}a_i&b_i\\ c_i&d_i\end{pmatrix}\in\Gamma$ for
  $i=1,2$. Let $L_i=\Z^2D_NV_i$ be the associated cocyclic
  lattices. Then the following are equivalent:
  \begin{enumerate}[(i)]
  \item $L_1=L_2$;
  \item $\Gamma_0(N)V_1=\Gamma_0(N)V_2$;
  \item $(c_1:d_1)_N=(c_2:d_2)_N$.
  \end{enumerate}
\end{lemma}
\begin{proof}
The equivalence of (ii) and~(iii) is standard: see, for example,
\cite[Proposition~2.2.1]{JCbook}.  For the equivalence of~(i)
and~(ii), we have $L_1=L_2\iff D_NV_1=UD_NV_2$ with~$U\in\Gamma$,
which is equivalent to~$V_1V_2^{-1}\in D_N^{-1}\Gamma D_N$.
\end{proof}

Since every element of~$\P(N)$ has the form $(c:d)_N$ with $c,d\in\Z$
coprime, there always exists a matrix~$V\in\Gamma$ with $(c\ d)$ as
its second row.  We call such a matrix a \emph{lift} of~$(c:d)_N$
to~$\Gamma$.  By~\Cref{lem:Gamma0N_cosets}, this lift is not unique,
but its $\Gamma_0(N)$-coset is.

In~\cref{eqn:LcdN-def}, we associated to each primitive
vector~$v=(c,d)$, and each integer~$N\ge1$, the set
\[
L(v;N) = L(c:d;N) = \{w\in\Z^2\mid w\wedge v\equiv0\pmod{N}\}.
\]
Note that~$v\in L(v;N)$ for all~$N\ge1$.
\begin{lemma}
  \label{lem:same-lattice}
  Let~$v_i=(c_i,d_i)\in\Z^2$ for~$i=1,2$ be primitive and~$N\ge1$. The
  following are equivalent:
  \begin{enumerate}[(i)]
  \item $v_1\wedge v_2\equiv0\pmod{N}$;
  \item $(c_1:d_1)_N = (c_2:d_2)_N$;
  \item $v_1\in L(v_2;N)$;
  \item $v_2\in L(v_1;N)$;
  \item $L(v_1;N) = L(v_2;N)$.
  \end{enumerate}
\end{lemma}
\begin{proof}
  Statements~(i)--(iv) are all equivalent to the homogeneous
  congruence~$c_1d_2\equiv c_2d_1\pmod{N}$, by definition. For~(v), it
  is enough to show that both contain the same primitive vectors,
  which follows immediately from the first parts.
\end{proof}

\begin{proposition}
  \label{prop:Lcdn-cocyclic}
  For all primitive~$v$ and all~$N\ge1$,
  \[
  L(v;N) = \Z^2D_NV \in \L(N),
  \]
  where $V\in\Gamma$ is any matrix lifting~$v$.  Moreover, the set of
  primitive vectors in~$L(v;N)$ is the set of all second rows of
  matrices in the coset~$\Gamma_0(N)V$.
\end{proposition}
\begin{proof}
  Let~$v\in\Z^2$ be primitive.  A simple computation shows that
  \begin{align*}
    w\in\Z^2D_NV &\iff w(D_NV)^{-1}\in\Z^2 \\
    &\iff v\wedge w\in N\Z \\
    &\iff w\in L(v;N).
  \end{align*}
  Hence~$L(v;N)=\Z^2D_NV$, and~$\Z^2D_NV$ is clearly a cocyclic lattice.

  For every primitive vector~$(c',d')\in L(v;N)$ we
  have~$(c:d)_N=(c':d')_N$, and~$(c',d')$ is the second row of a
  matrix~$V'\in\Gamma$, which by~\Cref{lem:Gamma0N_cosets} is in the
  same~$\Gamma_0(N)$-coset as~$V$.
\end{proof}

\begin{corollary}
  \label{cor:P1N-LcdN-bijection}
  Let~$N$ be a positive integer. The map
  \[
    (c:d)_N \mapsto L(c:d;N)
  \]
  (defined for coprime pairs~$c,d$) is a bijection
  from~$\P(N)$ to~$\L(N)$.
\end{corollary}

\begin{proof}
Immediate from~\Cref{lem:Gamma0N_cosets} and~\Cref{prop:Lcdn-cocyclic}.
\end{proof}

Recall that the cardinality of~$\P(N)$ is denoted~$\psi(N)$, which has
the formula~(\ref{eqn:psi-formula}).
\begin{corollary}
  For all~$N\ge1$,  $|\L(N)| = |\P(N)| = \psi(N)$.
\end{corollary}

The following easy facts will be used repeatedly in our discussion of
the lattice covering problem.

\begin{corollary}
  \label{cor:unique-LcdN}
  \begin{enumerate}[(i)]
    \item Every primitive vector~$v$ belongs to a unique cocyclic
      lattice of each index~$N$, namely~$L(v;N)$.
    \item The primitive vectors~$v$ in~$L(c:d;N)$ are precisely the
      lifts of~$(c:d)_N$ to~$\Z^2$, and for each lift~$v$ we
      have~$L(v;N)=L(c:d;N)$.
  \end{enumerate}
\end{corollary}

\begin{proof}
  Both statements follow from~\Cref{lem:same-lattice}.
\end{proof}

Note that the first statement is a restatement of the fact
that~$\Gamma$ is the disjoint union of the cosets of~$\Gamma_0(N)$.

By definition, a lattice~$L$ is cocyclic if it contains at least one
primitive vector.  More is true:
\begin{lemma}
  \label{lem:primitive-basis}
  Every cocyclic lattice has a basis of primitive vectors.  Hence
  every cocyclic lattice is uniquely determined by the primitive
  vectors it contains.
\end{lemma}
\begin{proof}
This is clear for the lattice $N\Z\oplus\Z$, for which such a basis is
given by $\{(N,1), (2N,1)\}$, and the general case follows by change
of basis in~$\Z^2$.
\end{proof}

Hence cocyclic lattices are uniquely determined either by their
index~$N$ and any one primitive vector~$v$ they contain,
as~$L(v;N)$ is the unique cocyclic lattice with this property, or by
the set of primitive vectors they contain.

If~$M,N$ are positive integers with $M\mid N$, then for
every~$L\in\L(N)$, there is unique~$L'\in\L(M)$ containing~$L$.  This
follows immediately from the fact that a cyclic group of order~$N$ has
a unique subgroup of index~$N/M$: the subgroup of~$\Z^2/L$ of index~$N/M$
is~$\Z^2/L'$.  If $v$ is any primitive vector in~$L$, then $L=L(v;N)$
and $L'=L(v;M)$.

\begin{lemma}
  \label{lem:new-lemma}
  Let $M,N\ge1$, let $L\in\L(M)$, and let $v$ be a primitive
  vector. Consider the statements
  \begin{enumerate}[(i)]
    \item $v\in L$;
    \item $L(v;N)\subseteq L$;
    \item $L$ contains all primitive vectors in $L(v;N)$.
  \end{enumerate}
  We have (ii)$\iff$(iii), and (ii)$\implies$(i). If, in addition,
  $M\mid N$, then also (i)$\implies$(ii), so that all three statements
  are equivalent.
\end{lemma}
\begin{proof}
(ii)$\iff$(iii) by the first statement in~\Cref{lem:primitive-basis},
  and (ii)$\implies$(i) is obvious; conversely, given (i), if~$M\mid
  N$, then~$L=L(v;M)\supseteq L(v;N)$.
\end{proof}

While the intersection of any two lattices is again a lattice, the
intersection of two cocyclic lattices may or may not be cocyclic.  We
say that two cocyclic lattices~\emph{intersect cyclically} if their
intersection is again cocyclic; otherwise they are~\emph{separated};
so two cocyclic lattices are separated if and only if their subsets of
primitive vectors are disjoint.

For example, the~$\psi(N)$ cocyclic lattices of index~$N$ are pairwise
separated, since by~\Cref{cor:unique-LcdN}, each primitive vector
belongs to exactly one of them.  At the other extreme, two cocyclic
lattices with coprime indices always intersect cyclically; this is
essentially a restatement of the homogeneous Chinese Remainder
Theorem~\Cref{prop:CRT-P1-full}.
\begin{proposition}\label{prop:CRT-lattice}
  Let~$N_1$, $N_2$ be coprime integers, and set $N=N_1N_2$.  There is
  a bijection between pairs~$(L_1,L_2)$ of lattices with
  indices~$[\Z^2:L_i]=N_i$ and lattices~$L$ of index~$N$, given by
  $(L_1,L_2)\mapsto L=L_1\cap L_2$. In this bijection, $L$ is cocyclic
  if and only if both the~$L_i$ are.

 If $L_i=L(c_i:d_i;N_i)$, then $L_1\cap L_2=L(c:d;N)$, where
 $(c:d)_N\in\P(N)$ is the common preimage of $(c_1:d_1;N_1)$
 and~$(c_2:d_2;N_2)$, which exists by~\Cref{prop:CRT-P1}.
\end{proposition}
\begin{proof}
  This is elementary group theory.  Given lattices~$L_i$ of
  index~$N_i$ respectively, the subgroup~$L_1L_2$ generated by
  the~$L_i$ is all of~$\Z^2$ since its index divides both~$N_i$.
  Hence, with~$L=L_1\cap L_2$,
  \begin{equation}
  \label{eq:decomp}
  \Z^2/L \cong \Z^2/L_1 \times \Z^2/L_2,
  \end{equation}
  so~$L$ has index~$N$.  For the reverse map, $L$ maps to~$(L_1,L_2)$
  where each~$L_i$ is the unique lattice of index~$N_i$
  containing~$L$: this exists and is unique since~$\Z^2/L$ is a cyclic
  abelian group of order~$N=N_1N_2$ so has unique subgroups of
  index~$N_1$ and~$N_2$.

  In the isomorphism~(\ref{eq:decomp}), $\Z^2/L$ is cyclic if and only if both
  factors on the right are, since their orders are coprime.

  The last part is clear.
\end{proof}

We now give a precise general condition for the intersection of two
cocyclic lattices to be cocyclic.

\begin{proposition}
\label{prop:CRT-lattice-full}
Let~$L_i\in\L(N_i)$ for~$i=1,2$. Set $M=\gcd(N_1,N_2)$
and~$N=\lcm(N_1,N_2)$.  Then the following are equivalent:
\begin{enumerate}[(i)]
\item $L_1$ and~$L_2$ are contained in the same lattice of index~$M$;
\item $L_1\cap L_2$ is cocyclic (that is, $L_1$ and~$L_2$ intersect
  cyclically).
\end{enumerate}
When these hold, the cocyclic intersection~$L_1\cap L_2$ has
index~$N$.

In more detail if~$L_i=L(c_i:d_i;N_i)$, then (i) holds if and only if
$(c_1:d_1)_M=(c_2:d_2)_M$, and then the intersection in (ii)
is~$L(c:d;N)$ where~$(c:d)_N$ reduces to~$(c_i:d_i)_{N_i}$
for~$i=1,2$.
\end{proposition}

\begin{proof}
This is immediate from~\Cref{prop:CRT-P1-full} and the bijection
between cocyclic lattices of index~$m$ and elements of~$\P(m)$
for~$m=N_1$, $N_2$, $M$, and~$N$.
\end{proof}

\begin{corollary}
  \label{cor:contain-or-separate}
Let~$L_i\in\L(N_i)$ for~$i=1,2$.  If~$N_1\mid N_2$ then
either~$L_2\subseteq L_1$, or~$L_1$ and~$L_2$ are separated.
\end{corollary}

As a special case of~\Cref{prop:CRT-lattice-full}, we see that given
any two cocyclic lattices $L_1=L(v_1;N)$ and $L_2=L(v_2;N)$ of the
same index, they are both contained in a lattice of index $M=\gcd(N,
v_1\wedge v_2)$, namely~$L(v_1,M)$. This is the same as~$L(v_2,M)$
since $v_1\wedge v_2\equiv0\pmod{M}$; it is the sum $L_1+L_2$, and is
the smallest lattice containing both~$L_1$ and~$L_2$. In the next
section we will need a similar construction of the sum of an arbitrary
number of lattices of the same index.
\begin{proposition}
\label{prop:gcd-lattices}
Let $N\ge1$, and for $1\le i\le k$ let $L_i=L(v_i,N)\in\L(N)$, where
the $v_i$ are primitive vectors.  Let
\[
M = \gcd(N, \{v_1 \wedge v_i : 1 \leq i \leq k\}),
\]
and set $L=L(v_1,M)$.  Then for all $i\le k$ we have $L=L(v_i;M)$
and~$L\supseteq L_i$, and for every cocyclic lattice~$L'$ we have
\[
L'\supseteq L \iff L'\supseteq L_i\ (\forall i).
\]
(That is, $L$ is the sum of all the $L_i$.)
\end{proposition}
Note that by transitivity of the relation $v\wedge w\equiv0\pmod{M}$
on primitive vectors, here we could have defined $M$ to be the $\gcd$
of~$N$ and $v_i\wedge v_j$ for \emph{all} $i\ne j$.
\begin{proof}
For all~$i$, since $M\mid v_1\wedge v_i$, we have $L(v_1;M)=L(v_i;M)$
(by~\Cref{lem:same-lattice}), and $L = L(v_i;M)\supseteq L(v_i;N)=L_i$ since
$M\mid N$. From these facts, the forward implication is clear.

For the backward implication, let~$L'$ be a cocyclic lattice of index~$M'$ containing
all the~$L_i$. Then $M'\mid N$, and $L'$ contains all the primitive
vectors~$v_i$, so $L'=L(v_i;M')$ for all~$i$.  Hence, for all~$i$, we
have $L(v_1;M')=L(v_i;M')$, so $v_1\wedge v_i\equiv0\pmod{M'}$.  So
$M'\mid M$, and hence $L'=L(v_i;M')\supseteq L(v_i;M)=L$ as claimed.
\end{proof}

\Cref{lem:index-p} and the bijection in~\Cref{cor:P1N-LcdN-bijection}
imply the following.
\begin{corollary}
  \label{cor:lattice-index-p}
  Let~$p$ be prime.  Each cocyclic lattice~$L$ of index~$N$ has $p+1$
  cocyclic sublattices of relative index~$p$ when~$p\nmid N$, and~$p$
  such sublattices when~$p\mid N$.

  If~$L=L(c:d;N)$, these are the lattices~$L(c_i:d_i;Np)$, where
  $(c_i:d_i)_{Np}$ run through the~$p$ or~$p+1$ lifts of~$(c:d)_N$
  from~$\P(N)$ to~$\P(Np)$.  Every primitive vector in~$L$ belongs to
  exactly one of these sublattices.
\end{corollary}

Note that while every lattice~$L$, being a free abelian group of
rank~$2$ and hence isomorphic to~$\Z^2$, always has exactly~$p+1$
sublattices of relative index~$p$, if $L$ is cocyclic of index
divisible by~$p$, one of these sublattices is not cocyclic.  For
example, $L(0:1;p) = p\Z\oplus\Z$ contains the non-cocyclic lattice
$p\Z^2$ with index~$p$, as well as the cocyclic lattices~$L(pt:1;p^2)$
for~$0\le t<p$.

\begin{corollary}
  \label{cor:index-p-intersection}
If~$L_i\in\L(N_i)$ for~$i=1,2$ intersect cyclically, and~$p$ is a prime
such that~$\ord_p(N_1)\ge\ord_p(N_2)$, then every cocyclic lattice of
index~$p$ in~$L_1$ also intersects~$L_2$ cyclically.
\end{corollary}
\begin{proof}
Immediate from~\Cref{prop:CRT-lattice-full}
since~$\gcd(N_1p,N_2)=\gcd(N_1,N_2)$.
\end{proof}


\section{Coverings, minimal coverings and the covering problem}
\label{sec:covering}
We recall the definition of a covering in~\Cref{def:covering}.  We start with a simple but important criterion for covering.

\begin{lemma}
\label{lem:primitive-cover}
A collection~$\C$ of lattices is a covering if and only if every
\emph{primitive} vector of $\Z^2$ is contained in one of the lattices
in~$\C$.
\end{lemma}

\begin{proof}
  Each lattice in~$\C$, and hence their union, is closed under taking
  scalar multiples, so if the union contains all primitive vectors
  then it contains all vectors.
\end{proof}

The following observation explains our focus on cocyclic lattices.

\begin{corollary}
  \label{cor:all-cocyclic}
  All lattices in an irredundant covering are cocyclic.
\end{corollary}
\begin{proof}
  If one lattice in the covering is not cocyclic, then it contains no
  primitive vectors.  So all primitive vectors of~$\Z^2$ are contained
  in the union of the remaining lattices, which thus form a cover by~\Cref{lem:primitive-cover}, contradicting irredundancy.
\end{proof}

\begin{remark}
   Hence the covering condition~(\ref{eqn:covering}) is equivalent to
   the statement that $\Gamma$ is the union of the corresponding
   cosets of $\Gamma_0(N_i)$, where $N_i=[\Z^2:L_i]$.
\end{remark}

In view of~\Cref{cor:all-cocyclic}, we need only consider coverings by
cocyclic lattices in what follows.  Since a set~$\{L(c_i:d_i;N_i)\}$
of cocyclic lattices covers~$\Z^2$ if and only if the associated
residue classes~$R(c_i:d_i;N_i)$ cover~$\P(\Z)$, we have a
correspondence between lattice coverings of~$\Z^2$ and the projective
covering congruence problem, covering~$\P(\Z)$ by projective (or
homogeneous) residue classes.

Correspondingly, a solution to the projective covering congruence
problem is strongly minimal if the associated residue classes are
disjoint, forming a partition of~$\P(\Z)$.

\begin{definition}
The \emph{size} of a covering~$\C$ is the cardinality~$|\C|$.
\end{definition}

The simplest construction of a lattice covering is the following:

\begin{definition}
For all~$N\ge1$, the set~$\L(N)$ of all cocyclic lattices of index~$N$
is a strongly minimal covering of size~$\psi(N)$, which we call
the~\emph{full index-$N$ covering}.
\end{definition}

That~$\L(N)$ is a strongly minimal covering follows
from~\Cref{cor:unique-LcdN}~(i).  Clearly, the only covering of size~$1$ is the \emph{trivial} covering
of~$\Z^2$ by itself.  We will see later (\Cref{prop:size-up-3}) that there are no coverings of
size~$2$, and only one of size~$3$, namely $\L(2)$.  The number of
coverings of size~$n$ grows quickly as a function of~$n$ (see~\cite{BBMST}, which covers the analogous case of classical covering systems).

\begin{example}
  The full index-$2$ covering is
  \[
  \L(2) = \{L(0:1;2), L(1:0;2), L(1:1;2)\}.
  \]
  Each primitive vector~$(x,y)$ belongs to exactly one of these,
  depending on whether $x$ is even, $y$ is even, or both are odd.

  Every nontrivial minimal covering must have size at least~$3$, since
  the vectors $(1,0)$, $(0,1)$, and~$(1,1)$ must belong to different
  lattices as each pair generates~$\Z^2$.  We will see later
  (\Cref{prop:size-up-3}) that this is the only covering of size~$3$.
\end{example}

\begin{remark}
  It is clear that, for every covering there is at least one
  irredundant subcollection which is still a covering: simply remove
  any lattice of the covering which is contained in the union of the
  others (if any), and repeat until there are no such redundant
  lattices.  This irredundant subcollection is not, in general,
  unique: for example, $\L(2)\cup\L(3)$ is a covering of size~$7$
  which contains at least two irredundant (and even minimal)
  irredundant coverings, namely~$\L(2)$ and~$\L(3)$.
\end{remark}

Let~$V(N)$ denote any set of primitive vectors of size~$\psi(N)$, one
from each lattice in~$\L(N)$.  Then~$\L(N)=\{L(v;N)\mid v\in V(N)\}$.
These sets (which are of course far from unique) may be used to test
whether a collection of lattices is a covering, and if so, whether the
covering is minimal, as we now demonstrate.  We start with a definition.

\begin{definition}
For a finite collection~$\C$ of cocyclic lattices, we define
its~\emph{index lcm} to be
\[
\lcm(\C)=\lcm(\{[\Z^2:L]\mid L\in\C\}).
\]
\end{definition}

\begin{proposition}
  \label{prop:covering-criterion}
  Let $\C$ be a collection of lattices, and $N\ge1$.
  \begin{enumerate}[(i)]
    \item If every~$L\in\L(N)$ is contained in at least one lattice
      in~$\C$, then $\C$ is a covering.
    \item Suppose that~$\lcm(\C)\mid N$.
      \begin{enumerate}
        \item If~$\C$ is a covering, then every~$L\in\L(N)$
          is contained in at least one lattice in~$\C$ (that is, the
          converse to (i) holds).
        \item If~$\C$ is not a covering, then for all primitive~$v\not \in\cup_{L\in\C} L$, the lattice~$L(v; N)$ is separated from all the
          lattices in~$\C$.
        \item If every~$v\in V(N)$ belongs to at
          least one lattice in~$\C$, then $\C$ is a covering.
      \end{enumerate}
  \end{enumerate}
\end{proposition}
\begin{proof}
For part~(i), every vector is contained in at least one lattice
in~$\L(N)$, since $\L(N)$ is a covering, and hence in at least one
lattice in~$\C$.

Now assume that~$\lcm(\C)\mid N$. Then,
by~\Cref{cor:contain-or-separate}, for every~$L\in\L(N)$ and~$L'\in\C$
we have (*): either $L\subseteq L'$ or $L,L'$ are separated.

For part~(ii)(a), let~$L\in\L(N)$.  If $L$ is not contained in any
lattice in~$\C$, then it is separated from each of them by (*); but then the
primitive vectors in~$L$ do not belong to any lattice in~$\C$, so~$\C$
is not a covering.

Part~(ii)(b) follows from~(*).  For
part~(ii)(c), we use~\Cref{lem:new-lemma} (namely (i) implies (ii)) and part~(i) of this lemma.
\end{proof}

Let~$\C$ be a covering. By definition, $\C$ is minimal if and only if
every lattice~$L\in\C$ is \emph{minimal in~$\C$}, meaning that for
cocyclic~$L'\subsetneq L$, $(\C\setminus L)\cup\{L'\}$ is not a
covering. We can check this condition for each~$L\in\C$, using a
simple numerical criterion also involving~$V(N)$ where~$N=\lcm(\C)$:
\begin{proposition}
  \label{prop:minimality-criterion}
Let~$\C$ be an irredundant covering with~$\lcm(\C)=N$, and
let~$L\in\C$ have index~$M$. Let~$S=\{v\in V(N)\mid v\notin L'
\quad(\forall L'\in\C, L'\not= L)\}$.  Then $S\not=\emptyset$, and
fixing one element~$v\in S$, set
\[
D = \gcd(N, \{v \wedge w : w \in S\})
\]
and $L'=L(v;D)$.  Then
\begin{enumerate}[(i)]
\item $L'$ does not depend on the choice of~$v\in
  S$.
\item $L'\subseteq L$.
\item $L$ is minimal in~$\C$ if and only if~$L'=L$, if and only if~$D=M$.
\item $\C'=(\C\setminus L)\cup\{L'\}$ is an irredundant covering such
  that~$L'$ is minimal in~$\C'$, with $|\C'|=|\C|$ and
  $\lcm(\C')=\lcm(\C)$.
\end{enumerate}
\end{proposition}
As remarked after~\Cref{prop:gcd-lattices}, in the definition of~$D$,
instead of fixing one~$v\in S$, we could more symmetrically
define~$D=\gcd(N, \{v \wedge w : v,w\in S\})$.
\begin{proof}
  For (i), we apply ~\Cref{lem:same-lattice}.  Indeed, since $D\mid v\wedge w$ by construction, this gives $L(v;D)=L(w;D)$ for all~$w\in S$, thus completing the proof of (i).


  Let~$L''\subseteq L$ be a cocyclic lattice, and consider~$\C'' =
  (\C\setminus L)\cup\{L''\}$. We claim that $\C''$ is a covering
  if and only if~$L''\supseteq L'$.

  Suppose first that~$\C''$ is a covering.  For~$w\in S$, set $L_w=L(w;N)$.  None of the lattices
  in~$\C\setminus\{L\}$ contain any~$w\in S$ (by definition of~$S$),
  and hence none contain any primitive vectors in any of the~$L_w$,
  by~\Cref{lem:new-lemma} (using the fact that all lattices in~$\C$
  have index dividing~$N$).  Then $L''$ must contain
  all primitive vectors in all the lattices~$L_w$ for~$w\in S$.  By~\Cref{lem:new-lemma}
  applied to~$L''$ (and specifically by the implication
  (iii)$\implies$(ii) of that lemma, which does not have any condition
  on the index of~$L''$) we see that~$L''\supseteq L_w$ for all~$w\in
  S$.  As~$S$ is not empty, ~\Cref{prop:gcd-lattices} implies
  that~$L''\supseteq L(v;D)=L'$.

  Conversely, suppose that~$L''\supseteq L'$. Since
  $L'=L(w;D)$ for all $w \in S$, we deduce that $w\in L''$ for all $w \in S$. Hence $S\subseteq
  L''$.  Also, the index of~$L''$ divides that of~$L'$, which is~$D$,
  and hence also divides~$N$.  By~\Cref{prop:covering-criterion}~(ii),
  $\C''$ is a covering.  We have now established the claim.

  Since~$\C$ itself is a covering, taking~$L''=L$ in the claim shows
  that~$L\supseteq L'$, proving~(ii), and also implying that~$M\mid
  D$.  Now it follows that~$L$ is minimal in~$\C$ if and only if~$L$
  is the only lattice~$L''$ satisfying~$L'\subseteq L''\subseteq L$,
  proving~(iii).

  Since~$M\mid D$ and~$D\mid N$, we see that~$\lcm(\C')=N$,
  where~$\C'=(\C\setminus L) \cup\{L'\}$.  By construction,~$\C'$ is a
  covering, and is clearly also irredundant.  Applying
  parts~(i)--(iii) to $\C'$ and~$L'$ in place of~$\C$ and~$L$ (so that
  $N$ and~$S$ are unchanged), we see that~$L'$ is minimal in~$\C'$,
  completing the proof of~(iv).
\end{proof}

As well as giving a criterion for a covering being minimal in
part~(iii), part~(iv) of~\Cref{prop:minimality-criterion} may be used
to give an algorithm which, starting with any irredundant covering,
successively replaces any lattices which are not minimal in the
current covering with a sublattice as prescribed as in part~(iv).  At
each stage, the size and~$\lcm$ of the covering remain the same.
After a finite number of steps this procedure results in a minimal
covering:
\begin{corollary}
  \label{cor:minimisation}
  Let~$\C$ be an irredundant covering.  Successively replacing each
  lattice~$L\in\C$ which is non-minimal in~$\C$ by the
  sublattice~$L'\subseteq L$ defined
  in~\Cref{prop:minimality-criterion}~(iv) results, after a finite
  number of steps, in a minimal covering~$\C_{\min}$, which
  satisfies~$|\C_{\min}|=|\C|$ and~$\lcm(\C_{\min})=\lcm(\C)$.  Each
  lattice in~$\C_{\min}$ is a sublattice of one of the lattices
  in~$\C$.
\end{corollary}
\begin{proof}
  The finiteness of the procedure follows from the fact that there are
  only finitely many cocyclic lattices of index
  dividing~$\lcm(\C)$. The rest is clear.
\end{proof}

\begin{remark}
In fact, one can see that any lattice which is minimal in the original
covering~$\C$ remains minimal throughout the procedure, so the number
of steps is bounded by the size of the original covering; we leave the
details to the reader.
\end{remark}

\begin{remark}
At each stage of this minimisation procedure, if there is more than
one non-minimal lattice, then a choice must be made as to which one to
replace with a minimal sublattice.  The minimal covering~$\C_{\min}$
constructed depends on these choices, so is not (in general) uniquely
determined by the original covering~$\C$. 
\end{remark}

\begin{remark}
  The literature concerning systems of covering congruences usually
  considers irredundant coverings rather than minimal ones in our
  sense.  \Cref{cor:minimisation} shows that this makes no difference
  when considering the relation between the size of a covering and its
  index~$\lcm$, as we do in~\Cref{sec:lower-bounds} below.  Our
  systematic enumeration of coverings in~\Cref{sec:small-size} only
  includes minimal coverings; it would be straightforward to extend
  this to also include non-minimal irredundant coverings.

  One property of coverings which is not preserved by the minimisation
  procedure of~\Cref{cor:minimisation} is that of having
  \emph{distinct indices}, which much of the literature on covering
  congruences is concerned with (following a question originally posed
  by~Erd\H{o}s).  Typically, these are not minimal.
\end{remark}

\begin{example}\label{ex:min-not-strong}
Not all minimal coverings are strongly minimal.  Consider
\[
\C = \{L(1:0;2), L(1:0;3), L(0:1;3)\} \cup \{L(c:1;6) \mid c \in \{\pm1,\pm2\}\}.
\]
This is a covering, of size~$7$: a vector~$(x,y)\in\Z^2$ is contained
in
\begin{itemize}
  \item $L(1:0;2)$ if and only if~$2\mid y$;
  \item $L(1:0;3)$ if and only if~$3\mid y$;
  \item $L(0:1;3)$ if and only if~$3\mid x$;
  \item $L(c:1;6)$ for~$c\in\{\pm1,\pm2\}$ if and only if $\gcd(x,3)=1$
    and $\gcd(y,6)=1$, where~$c$ satisfies $cy\equiv x\pmod{6}$.
\end{itemize}
It is not strongly minimal, since the primitive vector $(1,6)$ belongs
to both $L(1:0;2)$ and $L(1:0;3)$, but is minimal, as can be checked
using~\Cref{prop:minimality-criterion}.

In~\Cref{thm:all-upto-6} we will show that all minimal coverings of
size~$n\le6$ are strongly minimal, so this example of size~$7$ is a
smallest example of a minimal covering which is not strongly minimal.
\end{example}

The full index-$N$ coverings are not the only strongly minimal
coverings.  In the next section we define an operation called
\emph{refinement} which takes coverings to new coverings (consisting
of more lattices), and strongly minimal coverings to new strongly
minimal coverings.  While not all strongly minimal coverings may be so
obtained from the trivial covering of~$\Z^2$
(see~\Cref{ex:strong-not-refinement}), all minimal coverings of size
up to~$6$ are obtained this way (see~\Cref{thm:all-7}).

The following inequality and equality will be helpful in finding all
coverings of a given size~$n$, at least when~$n$ is small.

\begin{definition}
The \emph{weight} of~$L\in\L(N)$ is $\wt(L)=1/\psi(N)$. The \emph{weight} of a (finite) collection~$\C$ of cocyclic lattices is $\wt(\C)=\sum_{L\in\C}\wt(L)$.
\end{definition}

We may interpret the weight of a lattice~$L$ as a \emph{density}.  Recall that the density of a subset~$S\subseteq\Z^2$ is defined to be
the quantity~$\rho(S)$ given by following limit, if it exists:
\[
\rho(S) =
\lim_{B \rightarrow \infty} \frac{|\{ S  \cap [-B, B]^2\}|}{|\{\Z^2 \cap [-B, B]^2 \}|}.
\]
Denote the subset of primitive vectors in~$S \subseteq \Z^2$ by~$S_{\prim}$.  For example, it is well known that~$\rho(\Z^2_{\prim})=1/\zeta(2)$.  We also define the \emph{relative density}
\[
\rho_{\prim}(S) = \rho(S_{\prim})/\rho(\Z^2_{\prim}),
\]
of the primitive vectors in~$S$, relative to the density of all
primitive vectors.  

When $S$ is defined by congruence conditions, it is often the case
that~$\rho(S)$ is equal to an infinite product of local $p$-adic
densities, which are given for each prime~$p$ by the $p$-adic Haar
measure of the subset of~$\Z_p^2$ cut out by the same conditions.  A sufficient condition for this to happen is that the local conditions are \emph{admissible} as
defined in~\cite[Definition 4]{CremonaSadek}.  For example, a
vector~$(a,b)\in\Z^2$ is primitive if and only if for all primes~$p$
we have~$(a,b)\not\equiv(0,0)\pmod{p}$, and these conditions are
admissible (see~\cite[Example 3]{CremonaSadek}); the local density
is~$1-1/p^2$, and hence the global density is~$1/\zeta(2)$.

Changing finitely many local conditions does not affect admissibility:
this is Lemma~3.3 of~\cite{CremonaSadek}.  Hence we can
compute~$\rho(L_{\prim})$ for any cocyclic lattice~$L=L(c:d)\in\L(N)$
(where~$\gcd(c,d)=1$) by adding the condition~$ad\equiv bc\pmod{N}$ to
the primitivity condition, noting that the condition is vacuous except
at the finitely many primes dividing~$N$.  This allows us to prove the
following.
\begin{proposition}
  \label{prop:weight=density}
  Let~$L$ be a cocyclic lattice in~$\Z^2$ of index~$N$.  Then
  \[
  \rho_{\prim}(L) = 1/\psi(N) = \wt(L).
  \]
\end{proposition}
\begin{proof}
It suffices to show that $\rho(L_{\prim}) = 1/(\zeta(2)\psi(N))$. Let
$L=L(c:d;N)$ with $c,d$ coprime.  For each prime~$p$
let~$k=\ord_p(N)$. A simple calculation shows that the $p$-adic
density of $\{(a,b)\in\Z_p^2\mid (a,b)\not\equiv(0,0)\pmod{p},
ad\equiv bc\pmod{N}\}$ is $\varphi(p^k)/p^{2k} =
(1-1/p^2)/\psi(p^k)$. The result follows by the product formula.
\end{proof}

\begin{proposition}
  \label{prop:weight-formula}
Every covering~$\C$ satisfies the \emph{weight inequality}
\begin{equation}\label{eqn:weight-inequality}
  \wt(\C) \ge 1,
\end{equation}
with equality if and only if~$\C$ is strongly minimal.  Thus a
strongly minimal covering $\C$ satisfies the \emph{weight formula}
\begin{equation}\label{eqn:weight-formula}
  \wt(\C) = 1.
\end{equation}
\end{proposition}

\begin{proof}
Let~$\C$ be a strongly minimal covering.  Then, by definition of
strong primitivity, $\Z^2_{\prim}$ is the disjoint union of the
sets~$L_{\prim}$ for~$L\in\C$.  By the additivity of density we have
$1=\rho_{\prim}(\Z^2)=\sum_{L\in\C}\rho_{\prim}(L)=\wt(\C)$
by~\Cref{prop:weight=density}.

If $\C$ is not strongly minimal, then at least one pair~$L,L'\in\C$ has
cocyclic intersection, so the sets~$L_{\prim}$ and~$L'_{\prim}$ are
not disjoint. Their intersection~$(L\cap L')_{\prim}$ has relative
density~$\rho_{\prim}(L\cap L') = \wt(L\cap L')$ which is strictly
positive, and hence~$\wt(\C) \ge 1 + \wt(L\cap L') > 1$.
\end{proof}


For example, the full index-$N$ covering consists of~$\psi(N)$
lattices each of which has weight~$1/\psi(N)$, so the weight of this
strongly minimal covering is indeed~$1$.  Another way to
prove~\Cref{prop:weight-formula} without using the language of
densities is to compare the weight of a covering~$\C$ with that of the
strongly minimal covering~$\L(N)$, where~$N=\lcm(\C)$,
using~\Cref{prop:covering-criterion}. We leave the details to the
reader.

The weight inequality in~\Cref{prop:weight-formula} easily implies the
claims made earlier about minimal coverings of size at most~$3$:
\begin{proposition}
\label{prop:size-up-3}
  The only minimal covering of size~$1$ is the trivial covering (which
is the full index-$1$ covering).  There is no minimal covering of
size~$2$. The only minimal covering of size~$3$ is the full index-$2$
covering.
\end{proposition}
\begin{proof}
The case of size~$1$ is clear.  Minimal coverings of size~$n>1$
clearly cannot include~$\Z^2$ itself, so every lattice~$L_i$ in the
covering has index~$N_i\ge2$, with $\psi(N_i)\ge3$.  Now
(\ref{eqn:weight-inequality}) cannot hold with~$n=2$, and with $n=3$
it can only hold when all~$N_i=2$, so equality holds and the covering
is the full index-$2$ covering.
\end{proof}

The weight equation already implies a finiteness result for strongly
minimal coverings.  Proving that there are only finitely many irredundant 
coverings of size~$n$ is considerably harder:
see~\Cref{cor:finitely-many-minimal}.

\begin{theorem}
  The number of strongly minimal coverings of each size~$n$ is finite.
\end{theorem}
The proof follows from the weight equation~(\ref{eqn:weight-formula})
and the following two lemmas.
\begin{lemma}
  For fixed~$n\ge1$, and fixed~$T>0$, the
  equation~$\sum_{i=1}^{n}1/M_i=T$ has only finitely many solutions in
  positive integers~$M_1,\dots,M_n$.
\end{lemma}
\begin{proof}
We proceed by induction on~$n$.  If $n=1$, then there is one
solution~$M_1=1/T$ if~$1/T$ is integral, otherwise no solution.
Suppose~$n\ge2$, and let~$M=\min_i(M_i)$.  Then $1/M < \sum_i1/M_i \le
n/M$, and hence $1/T < M\le n/T$, so there are only finitely many
possibilities for~$M$.  For each of these, we may assume (permuting
the~$M_i$ if necessary) that $M_n=M$; then $\sum_{i=1}^{n-1}1/M_i =
T-1/M >0$, which has finitely many solutions by induction.
\end{proof}

\begin{lemma}
  For each~$M\ge1$, the equation~$\psi(N)=M$ has only finitely many
  solutions.
\end{lemma}
\begin{proof}
(This is essentially the same as the well-known proof that
  $\varphi(N)=M$ has only finitely many solutions, where~$\varphi$ is
  the Euler totient function.)

From the formula~(\ref{eqn:psi-formula}), we see that, for primes~$p$,
if~$p\mid N$ then $p+1\mid M$, so the primes possibly dividing~$N$
belong to a finite set.  Next, if $p^e\mid N$ then $p^{e-1}\mid M$,
which bounds the possible exponents of each prime.
\end{proof}
\begin{example}
Taking $M=12$, the only primes~$p$ such that $p+1\mid12$
are~$p=2,3,5,11$, and the maximal exponents are~$3,2,1,1$
respectively.  We find that $\psi(N)=12$ for $N=6,8,9,11$ and no more.
\end{example}

These lemmas are sufficiently explicit to be able to write a computer
program to list all solutions~$N_1,\dots,N_n$ to the weight
equation~(\ref{eqn:weight-formula}).  Such a program can be found in
our repository~\cite{CK}, though we do not use it for the results in
this paper.  Not all such solutions correspond to strongly minimal
coverings, and more conditions are needed to eliminate such
``impossible'' index sequences. The simplest of these is that no two
indices in a strongly minimal covering can be coprime (since the
lattices must be pairwise separated). Hence, for example, the index
sequence~$(2,3,3,4)$ has weight $1/3+1/4+1/4+1/6=1$, but there is no
strongly minimal covering with indices~$2,3,3,4$ as these are not
coprime.  Taking $n=4$, we find $11$ solutions to the weight equation
(up to permutation), but the only ones with coprime indices
are~$(2,2,4,4)$ and~$(3,3,3,3)$, both of which correspond to strongly
minimal coverings.


\section{Refinement of coverings}
\label{sec:refinement}
The process of refinement consists of replacing some or all of the
(cocyclic) lattices in a covering with several of its (cocyclic)
sublattices, in such a way that the covering property is preserved,
and if the original covering was strongly minimal before refinement,
then the refined covering is also strongly minimal.  Starting with the
trivial covering, which is strongly minimal, we thereby obtain more
strongly minimal coverings.


We will start with defining $p$-refinement and~$p$-descendants for a
prime~$p$.
\begin{definition}
  \label{def:p-descendants}
  The cocyclic sublattices of relative index~$p$ in a cocyclic
  lattice~$L$ are the \emph{$p$-descendants} of~$L$, and are called
  \emph{$p$-siblings} of each other.
\end{definition}

By \Cref{cor:lattice-index-p}, the number of $p$-descendants
of~$L\in\L(N)$ is $\psi(Np)/\psi(N)$; this is either~$p$ or~$p+1$,
depending on whether or not~$p\mid N$.  Hence the weight of~$L$ is
equal to the sum of the weights of its $p$-descendants.

As also stated in~\Cref{cor:lattice-index-p}, every \emph{primitive}
vector in a cocyclic lattice~$L$ belongs to \emph{exactly one} of its
$p$-descendants. Thus, the $p$-descendants are pairwise separated: the
set of primitive vectors of~$L$ is the disjoint union of the sets of
primitive vectors in its $p$-descendants.

\begin{example}
  Let $L=L(0:1;2)=\{(2x,y):x,y \in \Z\}$. The cocyclic sublattices of
  relative index~$2$ are~$L(0:1;4)$ and~$L(2:1;4)$, since the
  preimages of $(0:1)_2$ in $\P(4)$ are~$(0:1)_4$ and~$(2:1)_4$.
  Here, $L(0:1;4)=\{(4x,y)\}=\{(2x,y)\mid x\ \text{even}\}$ and
  $L(2:1;4)=\{(2x,y)\mid x\equiv y\pmod{2}\}$. The third sublattice of
  of index~$2$ in~$L$ is $\{(2x,y)\mid y\ \text{even}\}=2\Z^2$.  The
  imprimitive vector~$(2,2)$ is in~$L$ but is in neither of the first
  two sublattices.
\end{example}

\begin{definition}[Definition of $p$-refinement]
  \label{def:p-refinement}
  Given a covering $\C$, a lattice~$L\in\C$, and a prime~$p$, we can
  form a new covering by replacing~$L$ by all of its
  $p$-descendants~$L_j$, for~$1\le j\le m$, where~$m$ is either~$p$
  or~$p+1$ depending on whether or not the index of~$L$ is divisible
  by~$p$. This new covering
  \[
  \C' = \C \cup \{L_j\mid 1\le j\le m\} \setminus \{L\}
  \]
  is called a \emph{$p$-refinement} of~$\C$.
\end{definition}

%

\begin{example}
  The $p$-refinement of the trivial covering is the full index-$p$
  covering.
\end{example}

\begin{definition}[Definition of refinement]
  A \emph{refinement} of a covering is any covering obtained from the
original by applying any sequence of (zero or more) $p$-refinements,
using the same or different primes~$p$ at each stage.
\end{definition}

Note that refining a covering does give another covering, since the
set of primitive vectors in the union is unchanged.  Refining a
strongly minimal covering gives another strongly minimal covering (as
the weight has also not changed), and if a refinement of a minimal
covering is strongly minimal, then the original covering was also
strongly minimal (for the same reason).
In particular, we have the following.

\begin{theorem}
  \label{thm:triv-refine}
  All refinements of the trivial covering are strongly minimal.
\end{theorem}

We would like to bound the size of the covering in terms of the
indices of its lattices. This is straightforward for refinements of
the trivial covering, and the bound is expressed in terms of an
arithmetic function which we now define.

\begin{definition}
  The additive arithmetic function~$G$ is defined on prime powers~$p^e$
by~$G(p^e)=e(p-1)+1$, extended by additivity, so if~$N$ has prime
factorization~$N=\prod_{i=1}^{m}p_i^{e_i}$,
then~$G(N)=\sum_{i=1}^{m}(e_i(p_i-1)+1)$.

Note that $G(N)= F(N)+\omega(N)$, where~$F(N)$ is the totally additive
function\footnote{In~\cite{Simpson}, $F$ is denoted~$f$.}
with~$F(p)=p-1$ for a prime~$p$, and $\omega(n)$ is the usual function
counting the number of distinct prime factors of~$n$.
\end{definition}
The function~$G$ is additive but not totally additive:
\begin{lemma}
  \label{lem:G-additive}
  For all~$m,n\ge1$,
  \[
  G(m) + G(n) = G(l) + G(d),
  \]
  where~$d=\gcd(m,n)$ and~$l=\lcm(m,n)$.
\end{lemma}
\begin{proof}
  This is elementary, using~$G=F+\omega$, since~$F$ is totally
  additive and~$\omega$ clearly
  satisfies~$\omega(m)+\omega(n)=\omega(l)+\omega(d)$.
\end{proof}

The bound in the following theorem is easy to prove in the case of
refinement coverings (that is, refinement of the trivial covering).
In~\Cref{thm:Simpson-theorem-2-lattice} we will prove that the same
bound holds for all irredundant coverings.

\begin{theorem}
\label{thm:refinements-lower-bound}
  Let~$\C$ be a covering which is a refinement of the trivial
  covering. Let~$N=\lcm(\C)$.  Then
\[
|\C| \ge 1 + G(N).
\]
Equality holds if and only if, in the construction of~$\C$, every
$p$-refinement was applied to a lattice whose index had maximal
$p$-valuation.
\end{theorem}


\begin{proof}
  We proceed by induction on the number of refinement
  steps. Initially, for the trivial covering, we have $|\C|=1$ and
  $N=1$, so $G(N)=0$ and equality holds.

  Suppose that the result holds for a covering~$\C$ of size~$n$ with
  $N=\lcm(\C)$ and $g=G(N)$: so $n\ge g+1$, with equality if and only
  if the condition in the statement holds. Now apply $p$-refinement
  (for some prime~$p$) to one lattice~$L\in\C$, to obtain a new
  covering~$\C'$, and set~$n'=|\C'|$, $N'=\lcm(\C')$
  and~$g'=G(N')$. Let (*) be the condition that the index of~$L$ has
  the maximal~$p$-valuation of any lattice in~$\C$.

  If~$p\nmid N$, then (*) certainly holds, $n'=n+p$, $N'=pN$,
  and~$g'=g+p$.

  If~$p\mid N$ and (*) holds, then~$n'=n+p-1$, $N'=pN$, and~$g'=g+p-1$.

  Finally, if~$p\mid N$ and (*) does not hold, then either $n'=n+p$
  or~$n'=n+p-1$, while $N'=N$ and~$g'=g$.

  Hence, in all cases, we have $n'\ge g'+1$, and $n'=g'+1$ if and only
  if (*) holds and~$n=g+1$. This completes the induction.
\end{proof}


It is tempting to conjecture that every strongly minimal covering is a
refinement of the trivial covering, and this is true for strongly
minimal coverings of small size, but is not the case in general:
see~\Cref{ex:strong-not-refinement} for a counterexample of size~$29$.
However the conjecture is true in certain cases:
\begin{itemize}
\item every strongly minimal covering in which every lattice has index
  a power of one prime~$p$ is a refinement of the trivial covering
  (see~\Cref{thm:all-p-power});
\item every minimal covering of size~$n\le6$ is strongly minimal and
  is a refinement of the trivial cover; and every strongly minimal
  covering of size~$n\le8$ is a refinement of the trivial covering
  (see~\Cref{sec:small-size}).
\end{itemize}

Refining a minimal covering which is not strongly minimal does not
always lead to a minimal covering.
\begin{example}[Continuation of \Cref{ex:min-not-strong}]
  Consider again the minimal (but not strongly minimal) covering
  \[
  \C = \{L(1:0;2), L(1:0;3), L(0:1;3)\} \cup \{L(c:1;6) \mid c \in \{\pm1,\pm2\}\}.
  \]
As this is not strongly minimal, it is not a refinement of the trivial
covering.  If we apply~$2$-refinement to $L(1:0;3)$, replacing it
with~$L(1:0;6)$, $L(1:3;6)$, and~$L(2:3;6)$, we obtain a non-minimal
covering, since~$L(1:0;6)\subset L(1:0;2)$. This is an example where
refining a minimal covering leads to a redundant, hence non-minimal, covering.
\end{example}

We now present several constraints on minimal and strongly minimal
coverings, which we will use to determine all the minimal coverings of
sizes up to~$8$ in~\Cref{sec:small-size}, and further constraints on
strongly minimal coverings, which we will use to determine all the
strongly minimal coverings of sizes up to~$8$. The next result, which is~\Cref{thm:A}, states that coverings where the lattice indices have
only one prime divisor may be constructed by $p$-refinement of the
trivial covering.

\begin{theorem}
  \label{thm:all-p-power}
  Every irredundant covering~$\C$ with~$\lcm(\C)$ a prime power is a
  refinement of the trivial covering.
\end{theorem}
\begin{proof}
  We proceed by induction on~$n=|\C|$.  The case~$n=1$ is trivial.  Let~$\lcm(\C)=p^{k+1}$
  with~$k\ge0$, as otherwise~$\C$ is the trivial covering. 

  We will use the fact that for each primitive vector~$v$, the
  $p$-power index lattices~$L(v;p^e)$ containing~$v$ form an infinite
  chain, ordered by inclusion, with relative indices all equal to~$p$.

  Let~$L$ be a lattice in~$\C$ of maximal index~$p^{k+1}$, and let~$M$
  be the lattice of index~$p^k$ containing~$L$. We claim that all the
  $p$-descendants of~$M$ (of which~$L$ is one) belong to~$\C$.  If so,
  then~$\C$ is a $p$-refinement of a covering~$\C'$ in which all these
  descendants of~$M$ are replaced by~$M$ itself. As~$\C'$ has fewer
  lattices in it, all of which are of $p$-power index, we may conclude
  by induction provided that~$\C'$ is still irredundant.  Certainly,
  $M$ is not contained in the union of~$\C'\setminus\{M\}
  \subseteq\C\setminus\{L\}$, as if so, then the same would be true
  for~$L$, contradicting the irredundancy of~$\C$. Also,
  if~$L'\in\C'\setminus\{M\}$ is contained in the union
  of~$\C'\setminus\{L'\}$, this also contradicts the irredundancy
  of~$\C$. Indeed, since~$M$ contains exactly the same primitive vectors as
  the union of its $p$-descendants
  (\Cref{cor:lattice-index-p}), it would follow that all primitive vectors in~$L'$ are already contained in~$\C\setminus\{L'\}$, hence~$\C\setminus\{L'\}$ would contain all primitive vectors and thus be a cover by~\Cref{lem:primitive-cover}.

  To prove the claim, we must show that every primitive vector~$v\in
  M$ belongs to a lattice~$L_v\in\C$ of index~$p^{k+1}$. For each
  such~$v$, we have~$M=L(v;p^k)$ since~$v\in M$, and there is a
  lattice~$L_v$ in~$\C$ containing~$v$, with prime power index~$p^e$
  for some~$e\ge1$, so~$L_v=L(v;p^e)$.  By maximality of~$k$, we
  have~$e\le k+1$.  If $e\le k$, then $L\subseteq M = L(v;p^k)
  \subseteq L(v;p^e) = L_v$, contradicting the irredundancy
  condition. So $e=k+1$, and~$L_v$ is a $p$-descendant of~$M$ as
  required.
\end{proof}

\begin{corollary}
  \label{cor:n-p-e-bound-p-power}
Let $p$ be a prime, and let~$\C$ be an irredundant covering of
size~$n$ with~$\lcm(\C)=p^e$ (with~$e\ge1$).  Then~$n=k(p-1)+2$ for
some integer~$k\ge e$.  Hence $n\equiv 2\pmod{p-1}$, and if~$p\not=2$, then~$n$ is even.
\end{corollary}

\begin{proof}
By~\Cref{thm:all-p-power}, the covering is a refinement of the trivial
covering.
Let~$k$ be the number of
$p$-refinement steps, so~$k\ge e$. Now all $p$-refinement steps after
the first are applied to lattices of index divisible by~$p$, so the
first step increases the size by~$p$ and all others by~$p-1$, so~$n =
k(p-1)+2$. The last part is clear.
\end{proof}

Note that some condition of irredundancy is needed here, as otherwise we
could add to the covering a lattice of arbitrarily large $p$-power
index.

%
%
%

We now give an example of a strongly minimal covering which is not a
refinement of the trivial covering.

\begin{example}
  \label{ex:strong-not-refinement}
  First consider the pairwise separated lattices~$L_1=L(0:1;6)$,
  $L_2=L(1:1;10)$, and~$L_3=L(-1:1;15)$.  The set of refinements of
  these to index~$30$ consists of~$13$ distinct lattices,
  namely~$L(c:d;30)$ for the following~$(c:d)_{30}$:
  \begin{align*}
    (0:1)_6    &\mapsto (0:1)_{30}, (6:1)_{30}, (12:1)_{30}, (18:1)_{30}, (24:1)_{30}, (6:5)_{30};\\
    (1:1)_{10}  &\mapsto (1:1)_{30}, (11:1)_{30}, (21:1)_{30}, (1:21)_{30};\\
    (-1:1)_{15} &\mapsto (-1:1)_{30}, (14:1)_{30}, (13:2)_{30}.
  \end{align*}
  If we set $L_i$ for~$4\le i\le 62$ to be the other~$59$ cocyclic
  lattices of index~$30$ (of which there are~$\psi(30)=72$ in all),
  then we obtain a strongly minimal covering of size~$62$.  It is not
  a refinement of the trivial covering, since in any such refinement,
  all the lattice indices will be divisible by the prime used in the
  first refinement step, while here the indices are~$6$, $10$, $15$,
  and~$30$.

  This is not a minimal counterexample, but was chosen to have a
  relatively simple description.  One example of size only~$29$
  consists of:
  \begin{itemize}
  \item $L(c:1;15)$ for $0\le c \le14$;
  \item $L(1:0;6)$,  $L(2:3;6)$;
  \item $L(1:5;10)$;
  \item $11$ more lattices of index~$30$.
  \end{itemize}
  Here, the first~$18$ lattices have index strictly dividing~$30$ and
  are pairwise separated, and they refine to give~$61$ of the~$72$
  distinct cocyclic lattices of index~$30$; the last~$11$ complete the
  strongly minimal covering.
\end{example}


\section{Lower bounds for the size of a minimal covering}
  \label{sec:lower-bounds}
In this section we establish a lower bound for the size of a minimal
covering in terms of its index lcm.  Our treatment here follows that
in Simpson's paper~\cite{Simpson}, where the corresponding result for
systems of covering congruences is proved, but some new ideas are
required.

\begin{definition}
For $p$ prime, $k,i\in\Z$ with~$k\ge1$ and~$v\in\Z^2$ primitive, let
$P_{i,k}(v;p)$ be the distinct $p$-descendants of~$L(v;p^{k-1})$
\emph{excluding} $L(v;p^k)$, indexed by $1\le i\le p$ if $k=1$ and by
$1\le i\le p-1$ if~$k\ge2$.  Also, for primitive~$v$ and prime
power~$p^e$, define the collection
\[
\C(v;p^e)=\{P_{i,k}(v;p)\mid  k\le e\} \cup \{L(v;p^e)\}.
\]
\end{definition}

\begin{lemma}
  \label{lem:lemma-for-thm-1-lattice}
  \begin{enumerate}[(i)]
  \item
  For all primitive $v\in\Z^2$ and primes~$p$, the
  lattices~$P_{i,k}(v;p)$ for all $i$ (bounded as above) and
  all~$k\ge1$ are pairwise separated.
  \item
  For all primitive $v\in\Z^2$ and prime powers~$p^e$ (with~$e\ge1$),
  the collection $\C(v;p^e)$ is a strongly minimal covering of
  size~$G(p^e)+1$.  It is obtained from the trivial covering by
  applying $p$-refinement~$e$ times, at each step applying it to the
  unique lattice containing~$v$.
  \end{enumerate}
\end{lemma}

\begin{proof}
    For (i), we use the fact that two cocyclic lattices with the same index are
    either separated or equal (cf.~\Cref{cor:contain-or-separate}).  For fixed~$k$ we have
    \[
    P_{i,k}(v;p) = P_{j,k}(v;p) \iff i=j
    \]
    by definition.

    If $k_1>k_2$, then $P_{i,k_1}(v;p) \subseteq L(v;p^{k_1-1})
    \subseteq L(v;p^{k_2})$, therefore is separated from~$P_{j,k_2}(v;p)$
    for all~$j$, again by definition.

    For (ii), the construction by repeated $p$-refinement is clear:
    at the first step we obtain the full index-$p$ covering, and at
    the $k$th step we replace~$L(v;p^{k-1})$ with all its
    $p$-descendants, which include~$L(v;p^{k})$.  This implies
    that~$\C(v;p^e)$ is strongly minimal by~\Cref{thm:triv-refine},
    and the size of~$\C(v;p^e)$ is readily computed.
\end{proof}

The next theorem shows that these coverings~$\C(v;p^e)$ have the
smallest possible size among irredundant coverings containing a
lattice whose index is divisible by~$p^e$.  It is the analogue for
lattice coverings of ~\cite[Theorem 1]{Simpson}.

\begin{theorem}
  \label{thm:Simpson-theorem-1-lattice}
  Let~$\C$ be an irredundant lattice covering and let~$L=L(v;N)\in\C$ with
  $e=\ord_p(N)\ge1$.  Then
  \begin{enumerate}[(i)]
  \item For all~$i,k$ with $1\le k\le e$ and $1\le i\le p-1$ or $1\le
    i\le p$ (according as $k\ge2$ or $k=1$), there
    exists~$L_{i,k}\in\C$ with~$L_{i,k} \subseteq P_{i,k}(v;p)$.
    The~$L_{i,k}$ are pairwise separated, and separated from~$L$.
  \item $|\C| \ge 1+G(p^e)$, with equality if and only
    if~$\C=\C(v;p^e)$.
  \end{enumerate}
\end{theorem}
Note that the strongest result is obtained by taking $e$ to be the
largest exponent of~$p$ dividing any index of a lattice in the
covering.

\begin{proof}
    For (i), let $\B_k$ be the subcollection of~$\C$ consisting of those
    lattices with index not divisible by~$p^k$.

    First suppose that~$\B_k$ is empty; that is, all the lattices
    in~$\C$ have index divisible by~$p^k$.  For~$1\le i\le k$,
    let~$v_{i,k}$ be any primitive vector in~$P_{i,k}(v;p)$, and
    let~$L_{i,k}$ be a lattice in~$\C$ containing it.
    Then~$L_{i,k}=L(v_{i,k};N_{i,k})$ with~$p^k\mid N_{i,k}$,
    so~$L_{i,k}\subseteq P_{i,k}(v;p)$ by~\Cref{cor:contain-or-separate}, as both contain~$v_{i,k}$ and
    the index of~$P_{i,k}(v;p)$ is~$p^k$.

    Otherwise (when~$\B_k$ is not empty), let $M=\lcm(\B_k)$, so that
    $\ord_p(M)\le k-1$. Since~$L\not\in\B_k$ and~$\C$ is irredundant, the lattices
    in~$\B_k$ do not form a covering. In particular, by~\Cref{prop:covering-criterion}, there exists
    a primitive vector~$v_0\in L$ such that 
    $L':=L(v_0;M)$ is separated from any lattice in~$\B_k$.

    Since~$v_0 \in L \subseteq L(v;p^{k-1})$, the lattices $L'$ and
    $L(v;p^{k-1})$ both contain~$v_0$, so intersect cyclically.  Since
    $\ord_p(M) \le k-1$, $L'$
    also intersects cyclically each of the lattices~$P_{i,k}(v;p)$ by~\Cref{cor:index-p-intersection}, as
    these have index~$p$ in~$L(v;p^{k-1})$.

    Let~$v_{i,k}$ be a primitive vector in~$P_{i,k}(v;p)\cap L'$.
    Then~$L(v_{i,k};p^k) = P_{i,k}(v;p)$ (as both have index~$p^k$ and
    contain~$v_{i,k}$), and also~$L'=L(v_{i,k};M)$ (as both have
    index~$M$ and contain~$v_{i,k}$).  The latter implies
    that~$v_{i,k}$ is not in~$\cup_{L\in\B_k} L$, but since~$\C$ is a covering,
    there does exist a lattice $L_{i,k}$ in~$\C$ containing~$v_{i,k}$,
    with index~$N_{i,k}$ divisible by~$p^k$
    as~$L_{i,k}\notin\B_k$. Then
      \[
      L_{i,k}=L(v_{i,k};N_{i,k}) \subseteq L(v_{i,k};p^k) = P_{i,k}(v;p).
      \]
      The~$L_{i,k}$ are pairwise separated, as the~$P_{i,k}(v;p)$ are
      (see \Cref{lem:lemma-for-thm-1-lattice}).  Since~$L=L(v;N)\subseteq L(v;p^e)\subseteq
    L(v;p^k)$, each~$P_{i,k}(v;p)$ is separated from~$L$ by the second part of~\Cref{lem:lemma-for-thm-1-lattice} (with~$e$ replaced by~$k$),

      For (ii), we have constructed $1 + p + (e-1)(p-1) = 1+G(p^e)$
      distinct lattices in~$\C$, so~$|\C|\ge 1+G(p^e)$.  These
      lattices are each sublattices of one of the~$1+G(p^e)$ lattices
      in~$\C(v;p^e)$, so their total weight is at most~$1$, with
      equality if and only if they are exactly the lattices of that
      covering.  Thus~$|\C|=1+G(p^e)$ if and only if these are all the
      lattices in~$\C$. Since~$\C$ is a covering this is if and only
      if~$\C=\C(v;p^e)$.
\end{proof}

The following result is helpful when
enumerating lattice coverings, and gives a reason for some of the
patterns visible in the tables in~\Cref{sec:small-size}.

\begin{corollary}
  \label{cor:max-e-multiplicity}
  In an irredundant lattice covering~$\C$, suppose that for some
  prime~$p$ we have~$\ord_p(\lcm(\C))=e\ge1$. Then the number of
  lattices in~$\C$ whose index is divisible by~$p^e$ is at least~$p$
  when $e\ge2$ and is at least~$p+1$ when $e=1$.
\end{corollary}
\begin{proof}
In the notation of the theorem, all the~$L_{i,e}$ and~$L$ itself have
index divisible by~$p^e$; the number of these is~$p$ when~$e\ge2$
and~$p+1$ when~$e=1$.
\end{proof}

\begin{example}
  Consider a covering obtained from the trivial covering by first
  applying $2$-refinement and then $3$-refinement to one of the
  resulting index~$2$ lattices.  (Such a covering has
  type~$(2,2,(6,6,6,6))$ in the notation of~\Cref{sec:small-size}.)
  Two of the three lattices of index~$2$ are themselves in the
  covering, while the third contains (by construction) each of the
  index-$6$ lattices in the covering.  These four index-$6$ lattices
  are each contained in one of the four distinct index-$3$ lattices,
  again by the refinement construction.
\end{example}

The following is the lattice version of~\cite[Corollary 1]{Simpson}.

\begin{corollary}
  \label{cor:Simpson-cor-1-lattice}
  Let~$\C$ be an irredundant lattice covering and let~$L=L(v;d)\in\C$
  with $e=\ord_p(d)\ge1$ (as in~\Cref{thm:Simpson-theorem-1-lattice}).
  For~$1\le f\le e$, define the subcollection
  \[
  \C_0(f) = \{M\in\C : p^f\mid [\Z^2:M]\}.
  \]
  Then
  \begin{enumerate}[(i)]
  \item
    $|\C_0(f)| \ge G(p^e)-G(p^{f-1})+1$;
  \item
    if $L_s\in\L(p^e)$ for~$1\le s \le n$ are distinct, then the
    subcollection
    \[
    \C_0'(f) = \{M\in\C_0(f) \mid \text{$M$ is separated
      from~$L_s$ for~$1\le s\le n$}\}
    \]
    has size $|\C_0'(f)| \ge G(p^e)-G(p^{f-1})+1-n$.
  \end{enumerate}
\end{corollary}
\begin{proof}
  As in~\Cref{thm:Simpson-theorem-1-lattice}, $\C$ contains
  lattices~$L_{i,k}$ for $f\le k\le e$ and $1\le i\le p-1$ or
  (when~$f=k=1$) $1\le i\le p$; they all belong to~$\C_0(f)$, and the
  number of them is~$G(p^e)-G(p^{f-1})$.  Also, $L\in\C_0(f)$,
  since~$f\le e$.  Let~$\C_1(f)$ be the subcollection of~$\C_0(f)$
  consisting of the~$L_{i,k}$ (for~$i,k$ as above) together with~$L$.
  Then~$|\C_0(f)| \ge |\C_1(f)| = G(p^e)-G(p^{f-1}) +1$, giving~(i).

  For (ii), we claim that each of the~$n$ lattices~$L_s$ can intersect
  cyclically at most one of the lattices in~$\C_1(f)$. If~$L_s$ is not
  separated from~$L_{i,k}$ for some~$s\le n$, then~$L_s$ is also not
  separated from~$P_{i,k}$ (since~$L_{i,k}\subseteq P_{i,k}$),
  so~$L_s\subseteq P_{i,k}$ by~\Cref{cor:contain-or-separate}, since
  their indices are~$p^e$ and~$p^k$ with~$e\ge k$.  Similarly, if
  $L_s$ is not separated from~$L=L(v;d)\subseteq L(v;p^e)$, then~$L_s
  = L(v;p^e)$.  Hence the claim follows
  from~\Cref{lem:lemma-for-thm-1-lattice}~(ii).

  From the claim, we deduce that~$|\C_1(f)\setminus(\C_0'(f) \cap
  \C_1(f))|\le n$. Therefore the result follows from the bounds
  $|\C_0'(f)| \geq |\C_0'(f) \cap \C_1(f)| \geq |\C_1(f)| - n$.
\end{proof}

The lattice analogue of Simpson's Theorem~2 in~\cite{Simpson} requires
a result to take the place of Simpson's Lemma~3.  For this we need to
set up a correspondence between coverings of~$\Z^2$ with index
lcm~$N$, and coverings of a fixed cocyclic lattice~$L$ of index~$M$,
coprime to~$N$: by a~\emph{covering of~$L$} we mean a collection of
sublattices of~$L$ whose union is~$L$.  The definitions of minimality,
strong minimality and size of such relative coverings extend to this
more general situation in an obvious way.  Note that
since~$\gcd(M,N)=1$ here, by~\Cref{prop:CRT-lattice}, the condition
that a sublattice~$L'\subseteq L$ of relative index~$N$ be cocyclic
(meaning, as always, that~$\Z^2/L'$ is cyclic) is equivalent to~$L/L'$
being cyclic. We call such a sublattice of~$L$ a ``cyclic sublattice''
of~$L$.

The criterion of \Cref{lem:primitive-cover} extends to this relative situation:
\begin{lemma}
  \label{lem:relative-primitive-cover}
  Let~$L\in\L(M)$, and let $\C=\{L_1,\dots,L_n\}$ be a collection of
  cyclic sublattices of~$L$ with relative indices $N_i=[L:L_i]$ all
  coprime to~$M$. Write each $L_i=L\cap L_i'$ with $L_i'\in\L(N_i)$,
  and set $\C'=\{L_1',\dots,L_n'\}$.  Then the following are
  equivalent:
  \begin{enumerate}[(i)]
  \item $\C'$ covers $\Z^2$;
  \item $\C$ covers $L$;
  \item every primitive vector in~$L$ belongs to at least one of the~$L_i$.
  \end{enumerate}
\end{lemma}
\begin{proof}
That (i) $\implies$ (ii) $\implies$ (iii) is trivial, so assume (iii).
Note that while every vector in~$L$ is a multiple of a primitive
vector, that primitive vector need not also belong to $L$; this is why
the proof of~\Cref{lem:primitive-cover} will not work directly here.

Let $N=\lcm(\C')$, which by hypothesis is coprime to~$M$. To show
that $\C'$ covers~$\Z^2$, it suffices to show that every
lattice~$L'\in\L(N)$ is contained in one of the~$L_i'$.  By the
coprimality of the indices, $L'\cap L$ is cocyclic and so contains a
primitive vector $v$. By (iii), $v\in L_i$ for some~$i$, and $L' =
L(v;N) \subseteq L(v;N_i) = L_i'$.
\end{proof}

\begin{proposition}
  \label{prop:relative-coverings}
Let $L\in\L(M)$.  Then there is a bijection between
\begin{enumerate}
\item the set of coverings $\C$ of $L$ by cocyclic lattices of
  relative index coprime to~$M$, and
\item the set of coverings $\C'$ of $\Z^2$ by lattices of absolute
  index coprime to~$M$,
\end{enumerate}
given by mapping~$\C'\mapsto\C$ via the bijection~$L'\mapsto L\cap L'$.  Moreover, $\C$ is irredundant if and only if $\C'$ is.
\end{proposition}

\begin{proof}
  Given~$\C'$ as in~(2), set~$\C=\{L\cap L'\mid L'\in\C'\}$. Since the
  indices of~$L$ and~$L'$ for~$L'\in\C'$ are coprime, each~$L\cap L'$
  is cocyclic, each is contained in~$L$, and~$\C$ is a covering
  of~$L$ since every~$v\in L$ belongs to at least one~$L'\in\C'$.

  Conversely, given~$\C$ as in~(1), write each~$L_i\in\C$ as
  $L_i=L\cap L_i'$ (as in~\Cref{prop:CRT-lattice}) where
  $[\Z^2:L_i']=[L:L_i]$ is coprime to~$M$, and set $\C'=\{L_i'\mid
  L_i\in\C\}$. This is a covering of~$\Z^2$
  by~\Cref{lem:relative-primitive-cover}.

  It is clear that these maps are inverses of each other, and that
  $\C$ is irredundant if and only if $\C'$ is.
\end{proof}

Now we come to the main result of this section.  It is the analogue
for lattices of Simpson's Theorem~2, and is precisely our~\Cref{thm:B}.

\begin{theorem}
  \label{thm:Simpson-theorem-2-lattice}
  Let~$\C$ be an irredundant covering, $N=\lcm(\C)$, and let~$D\mid
  N$, $D\not=N$.  Then
  \[
  |\{L\in\C : [\Z^2:L]\nmid D\}| \ge 1 + G(N) - G(D).
  \]
  In particular, if~$\C$ is not the trivial covering, then~$N>1$, and
  taking~$D=1$ gives
  \[
  |\C| \ge 1 + G(N).
  \]
\end{theorem}

\begin{proof}
The proof is by induction on~$\omega(N)$, the number of distinct
primes dividing~$N$.  The case $N = 1$ is true vacuously.

If~$\omega(N)=1$, then~$N=p^e$ and~$D=p^k$ with~$p$ prime
and~$e>k\ge0$.  Note that~\Cref{thm:Simpson-theorem-1-lattice} already
gives~$|\C|\ge G(N)+1$ in this case, but we need the version with
general~$D$ for the inductive step later.
By~\Cref{cor:Simpson-cor-1-lattice} with~$f=k+1$ and~$n=0$, we have
\[
|\{L\in\C: [\Z^2:L]\nmid D\}| = |\{L\in\C: p^f\mid[\Z^2:L]\}|
\ge G(N)-G(D)+1.
\]

For the inductive step, we change notation so that~$\lcm(\C)=p^eN$,
where~$p\nmid N$ and~$e\ge1$, and the inductive hypothesis is that the
result holds for irredundant coverings with at most~$\omega(N)$ primes
in their index lcm.  Let~$p^fD$ be a divisor of~$p^eN$
with~$p^fD\not=p^eN$, where~$D\mid N$ and either~$f<e$ or~$D<N$ (or
both).  Let
\[
\C^* = \{L\in\C : [\Z^2:L]\nmid p^fD\}.
\]
To complete the induction we must show that
\[
|\C^*| \ge G(p^eN) - G(p^fD) + 1 = (G(N)-G(D)) + (G(p^e)-G(p^f)) + 1;
\]
here we have used the additivity of~$G$.

Every lattice in~$\C$ has the form~$L=L(v;p^kM)$ where $k\le e$
and~$M\mid N$ (so~$p\nmid M$), so is the intersection $L = L'\cap
L''$, where $L'=L(v;M)$ and $L''=L(v;p^k)$ are cocyclic lattices with
indices~$M$ (coprime to~$p$) and~$p^k$ respectively.  We will use this
association and notation $L\leftrightarrow(L',L'')$ repeatedly below.

Enumerate the lattices in~$\L(p^e)$ as~$\L(p^e)=\{L_s\mid 1\le s\le
t\}$ where~$t=\psi(p^e)$.  
For each~$s$ with $1\le s
\le t$, we will use~$\C$ to construct an irredundant covering~$\C_s'$
of~$\Z^2$ with $\lcm(\C_s')\mid N$.  We then apply the inductive
hypothesis to certain of these coverings.

Fix $s$ with $1\le s\le t$.  For~$L\in\C$, writing~$L=L'\cap L''$ as
above, we have~$L\cap L_s = L'\cap(L''\cap L_s)$. Since~$L'$ has index
coprime to~$p$, this is cocyclic if and only if~$L''\cap L_s$ is,
which is if and only if~$L''\supseteq L_s$
by~\Cref{cor:contain-or-separate}, and in this case, $L\cap L_s=L'\cap
L_s$.

Define
\begin{align*}
\A_s &= \{L\in\C\mid L\cap L_s\ \text{is cocyclic}\}; \\
\A_s' &= \{L'\mid L=L'\cap L'', L\in\A_s\}; \\
\A_s'' &= \{L'\cap L_s\mid L=L'\cap L'', L\in\A_s\}.
\end{align*}
There is a natural map~$\phi: \A_s \rightarrow \A_s'$ given by~$L=L'\cap L'' \mapsto L'$ and there is a natural map~$\rho: \A_s' \rightarrow \A_s''$ given by~$L' \mapsto L'\cap L_s$.  Note that~$\phi$ is clearly surjective and that~$\psi$ is a bijection by~\Cref{prop:CRT-lattice}.  Also note that, in the definition of~$\A_s'$, the
lattice~$L\in\C$ such that~$L=L'\cap L''$ is uniquely determined
by~$L'$ (where, as always, the index of~$L'$ is coprime to~$p$ and
that of~$L''$ is a power of~$p$), by the irredundancy of~$\C$ and the
fact that the lattices containing~$L_s$ are totally ordered by
inclusion.  It follows that~$\phi$ and~$\rho$ give bijections~$\A_s \leftrightarrow \A_s'
\leftrightarrow \A_s''$.

We claim that~$\A_s''$ is a covering of~$L_s$ by cyclic
sublattices. First, each lattice in~$\A_s''$ is (by definition) a
sublattice of~$L_s$ with relative index coprime to~$p$.  Let~$v\in L_s$
be primitive.  Since~$\C$ is a covering, $v\in L$ for some~$L\in\C$,
and~$L\cap L_s$ is cocyclic (as it contains~$v$),
so~$L\in\A_s$.  Then~$v\in L\cap L_s=L'\cap L_s \in\A_s''$.  By~\Cref{lem:relative-primitive-cover}, it follows that~$\A_s''$ is a covering of~$L_s$ and that~$\A_s'$ is a covering of~$\Z^2$.

However, the coverings~$\A_s'$ of~$\Z^2$ and~$\A_s''$ of~$L_s$ may
not be irredundant.  In order to apply the inductive hypotheses, we let~$\C_s'$ be an irredundant subcovering of $\A_s'$. We define~$\C_s'' := \rho(\C_s')$ and~$\C_s := \phi^{-1}(\C_s')$.  By~\Cref{prop:relative-coverings}, $\C_s''$ is an irredundant covering of~$L_s$.

Set $N_s=\lcm(\C_s')$. Then $N_s \mid N$, and~$|\C_s| = |\C_s'|$.

We next claim that the union of~$\{\C_s\mid 1\le s\le t\}$, which is
certainly a subcollection of~$\C$, is a covering.  The irredundancy
of~$\C$ then implies that~$\C = \bigcup\{\C_s\mid 1\le s\le t\}$.  Let~$v \in \Z^2$. We have~$v\in L_s$ for some~$s$, since the collection of all
the~$L_s$ is a covering.  Then, since~$\C_s''\subseteq\A_s''$ is a
covering of~$L_s$, there exists (by definition of~$\A_s''$)~$L'\cap L_s = L\cap L_s \in \C_s''$ with~$L\in\A_s$ such that~$v\in L'\cap L_s = L\cap L_s$.  Moreover, by definition of~$\phi$ and~$\rho$, we see that~$L' \cap L_s \in \C_s''$ implies that~$L \in \C_s$.

Hence $N=\lcm(N_1,N_2,\dots,N_t)$.

Now set $R_0=D$, and for~$s=1,2,\dots, t$ in turn define the following:
\begin{itemize}
\item $R_s=\lcm(R_{s-1},N_s)$; then~$R_{t}=\lcm(D,N_1,N_2,\dots,N_t)=N$.
\item $D_s=\gcd(R_{s-1},N_s)$, so $R_s/R_{s-1}=N_s/D_s$. Note that
  \begin{equation}
    \label{eqn:G-additive}
  G(R_s) - G(R_{s-1}) = G(N_s) - G(D_s)
  \end{equation}
  by~\Cref{lem:G-additive}.
\item the collection
\begin{align*}
  \B_s
  &=
  \{L = L'\cap L''\in\C_s :  \text{$[\Z^2:L']=M$ with $M\nmid  D_s$}\} \\
  &=
  \{L = L'\cap L''\in\C_s : \text{$[\Z^2:L']=M$ with $M\nmid R_{s-1}, M\mid R_s$}\}.
\end{align*}
  Equality holds here because, firstly, for all~$L = L'\cap
  L''\in\C_s$ with~$[\Z^2:L']=M$ we have~$M\mid R_s$, since $M\mid
  N_s$ (by definition of~$N_s$) and $N_s\mid R_s$ (by definition
  of~$R_s$); and secondly, since~$D_s=\gcd(R_{s-1},N_s)$, we
  have~$M\mid R_{s-1}\iff M\mid D_s$.
\end{itemize}
The second expression for~$\B_s$ shows that the collections~$\B_s$ are
pairwise disjoint, since $R_{s-1}\mid R_s$ for all~$s$, so a lattice
in~$\B_s$ cannot belong to any~$\B_t$ for~$t>s$.  Some of the~$\B_s$
may be empty; in fact this certainly happens when~$D_s=N_s$, for
then~$R_s=R_{s-1}$.  Let~$S=\{s\mid D_s\not=N_s\}$, and $n=|S|$.

As~$\B_s\subseteq\C_s$, for each~$L\in\B_s$ we have $L\cap L_s =
L'\cap L_s$ where $[\Z^2:L']$ is coprime to~$p$, and the
subcollection~$\B_s'\subseteq\C_s'$ consisting of these lattices~$L'$
is
\[
\B_s' = \{L'\in\C_s' \mid [\Z^2:L'] \nmid D_s\}.
\]
Since $\lcm(\C_s')=N_s$, and $D_s$ is a proper divisor of~$N_s$
for~$s\in S$, and as the covering $\C_s'$ is irredundant, we can apply
the induction hypothesis to $\C_s'$, $N_s$ and~$D_s$ for~$s\in S$ to
obtain
\[
|\B_s| = |\B_s'| \ge G(N_s) - G(D_s) +1 \qquad\text{for all~$s\in S$}.
\]
Adding over all~$s\in S$, using~(\ref{eqn:G-additive}) and the fact
that the~$\B_s$ are pairwise disjoint, gives
\begin{align} \label{eqn:lb-Bs}
  \left|\bigcup_{1\le s\le t} \B_s \right| = \sum_{s\in S}|\B_s|
  & \ge \sum_{s=1}^{t}(G(N_s)-G(D_s))+n \notag \\
  & = \sum_{s=1}^{t}(G(R_s)-G(R_{s-1}))+n \notag \\
  & = G(N) - G(D) + n,
\end{align}
since~$R_{t}=N$ and~$R_0=D$.

Recall that
\[
\C^* = \{L\in\C : [\Z^2:L]\nmid p^fD\}.
\]
Write~$\C^* = \C_1^* \cup \C_2^*$, where
\[
\C_1^* = \{ L\in\C : p^{f+1}\mid[\Z^2:L]\},
\]
and
\[
\C_2^* = \{ L\in\C : \text{$[\Z^2:L]=p^kM$ with~$p\nmid M$ and~$M\nmid D$}\}.
\]
These may not be disjoint, but we will identify disjoint subsets of
each and find lower bounds for the size of these subsets to obtain a
lower bound for the size of~$\C^*$.

First consider~$\C_1^*$.  By definition of~$\C_s$ (which, we recall,
is a subcollection of~$\A_s$), the lattices in~$\C_s$ are not separated from~$L_s$. Hence
\begin{align*}
  \C_1^* &\supseteq \C_1^* \setminus\cup_{s\in S}\C_s \\
  &\supseteq \{ L\in\C:\text{$p^{f+1}\mid [\Z^2:L]$
    and $L$ is separated from~$L_s$ for all $s\in S$}\}\\
  &=\C_1^{**}\quad\text{say},
\end{align*}
and so by~\Cref{cor:Simpson-cor-1-lattice} we have
\begin{equation}
  \label{eqn:C1s-bound}
  |\C_1^*|\ge|\C_1^{**}| \ge G(p^e)-G(p^f)+1-n.
\end{equation}

Secondly, note that if~$L\in\C_s$ has index~$p^kM$ with~$p\nmid M$
and~$M\mid D$, then~$M\mid R_{s-1}$, so~$L\notin\B_s$.  Hence
\begin{align*}
\C_2^* &\supseteq \cup_{s\in S}\{ L = L'\cap L''\in\C_s : \text{$[\Z^2:L']=M$
  with~$M\nmid D$}\} \\ &\supseteq\cup_{s\in S}\B_s.
\end{align*}
Together with~(\ref{eqn:lb-Bs}), this gives
\begin{equation}
  \label{eqn:C2s-bound}
|\C_2^*| \ge \sum_s|\B_s| \ge G(N) - G(D)+n.
\end{equation}

Adding~(\ref{eqn:C1s-bound}) and~(\ref{eqn:C2s-bound}), noting
that~$\C_1^{**}$ and~$\cup_{s\in S}\B_s$ are disjoint, we obtain
\[
|\C^*| \ge G(N)-G(D) + G(p^e)-G(p^f) + 1 = G(p^eN)-G(p^fD)+1
\]
as required.
\end{proof}

\begin{corollary}
  \label{cor:finitely-many-minimal}
  For each~$n\ge1$, the number of irredundant coverings of size~$n$ is
  finite.
\end{corollary}

\begin{proof}
  For a fixed size~$n$, \Cref{thm:Simpson-theorem-2-lattice} bounds
  both the primes~$p$ which can divide the indices in the covering and
  the maximal exponent of each~$p$. Hence the lattice indices are
  bounded. As there are only finitely many lattices of each
  index, the result follows.
\end{proof}


\section{Lattice coverings of small size}
\label{sec:small-size}
In this section we determine all minimal coverings of size up to~$8$.
For sizes~$7$ and~$8$, the completeness of our lists of coverings
relies on the use of a computer program~\cite{CK}, based on the same algorithm
as a similar program written by the second author which was used to
find all minimal coverings of size up to~$6$ for his
papers~\cite{FouvryKoymansI}, \cite{FouvryKoymansII},
and~\cite{FouvryKoymansIII}, which the first author reimplemented
(with some efficiency improvements based on the theory developed
above) in~\Sage~\cite{Sage}.

In~Tables\ref{tab:size-upto-6}, \ref{tab:size-7}, and~\ref{tab:size-8}
(see Appendix~A), we list, in abbreviated form, all minimal coverings
of size up to~$8$.

\subsection{Additional constraints on minimal coverings}
The following two results will allow us to determine all minimal
coverings of size up to~$6$.

\begin{proposition}
  \label{prop:small-n-index-bounds}
  Let~$\C$ be an irredundant covering of size~$n$.
  \begin{enumerate}[(i)]
  \item
    If $n\le5$, then~$\lcm(\C) \in\{1,2,3,4,8\}$, and~$\C$ is a
    refinement of the trivial covering.
  \item
    If $6\le n\le10$, then~$\lcm(\C)$ has at most~$2$ prime factors.
  \end{enumerate}
\end{proposition}

\begin{proof}
  If~$\lcm(\C)$ has at least two (distinct) prime factors~$p,q$,
  then we have
  $$
  |\C|\ge p+q+1\ge6
  $$
  by~\Cref{thm:Simpson-theorem-2-lattice}. So if $n\le5$, then $\lcm(\C)$ is a prime
  power.  Similarly, if~$\lcm(\C)$ has three prime factors~$p,q,r$,
  then~$|\C|\ge p+q+r+1\ge 2+3+5+1=11$, implying~(ii).

  For~$n\le5$ with~$\lcm(\C)=p^e$, we have $5\ge n\ge e(p-1)+2 \ge
  p+1$, so~$p=2$ or~$3$ and $e\le3$ or $e=1$
  respectively. By~\Cref{thm:all-p-power}, $\C$ is a refinement of the
  trivial covering. This completes the proof of~(i).
\end{proof}

The first part of the next result is needed to complete the
determination of minimal coverings of size~$6$; the other two parts
give rise to additional minimal coverings of size~$7$ and~$8$.

\begin{proposition}
  \label{prop:lcm-pq}
  Let~$p$ and~$q$ be distinct primes, and~$\C$ a minimal covering
  with~$\lcm(\C)=pq$, so that~$|\C|\ge p+q+1$.
  \begin{enumerate}[(i)]
  \item
    If~$|\C|=p+q+1$, then~$\C$ is obtained either by applying
    one~$q$-refinement to~$\L(p)$ or one $p$-refinement to~$\L(q)$.
    There are~$p+q+2$ minimal coverings of this type, all strongly
    minimal.
  \item
    If~$|\C|=p+q+2$, then~$\C$ consists of all but two of the lattices
    in~$\L(p)$, and all but two of the lattices in~$\L(q)$, together
    with four from~$\L(pq)$.  There are~$\binom{p+1}{2}\binom{q+1}{2}$
    minimal coverings of this type, none strongly minimal.
  \item
    $|\C|=pq+2$ for all coverings~$\C$ obtained by taking one lattice
    of index~$p$, one of index~$q$, and~$pq$ of index~$pq$. There
    are~$(p+1)(q+1)$ of these, none strongly minimal.
  \end{enumerate}
\end{proposition}
\begin{proof}
    For (i), suppose $\lcm(\C)=pq$ and~$|\C|=p+q+1$. Let~$n_p$
    (respectively~$n_q$) be the number of lattices in~$\C$ with
    index~$p$ (respectively~$q$).  Together, these cover all
    but~$(p+1-n_p)(q+1-n_q)$ of the~$(p+1)(q+1)$ lattices in~$\L(pq)$,
    so this must be the number of lattices in~$\C$ of index~$pq$.  Hence
    \[
    p+q+1-n_p-n_q = (p+1-n_p)(q+1-n_q),
    \]
    which simplifies to~$(p-n_p)(q-n_q)=0$. Hence either~$n_p=p$
    or~$n_q=q$.

    If~$n_p=p$, then we must have~$n_q=0$: any lattice~$L$ in~$\C$ of
    index~$q$ would already have all but one of its~$p$-descendants
    covered by the lattices of index~$p$, so could be replaced by the
    remaining~$p$-descendant, contradicting minimality. So in this case
    $\C$ consists of~$p$ lattices of index~$p$ together with all
    the~$q$-descendants of the remaining index~$p$ lattice.  There
    are~$p+1$ such coverings, depending on which of the index~$p$
    lattices has been~$q$-refined.

    Similarly if~$n_q=q$, changing the roles of~$p$ and~$q$.

    For (ii), a similar argument gives~$(p-n_p)(q-n_q)=1$, so either~$n_p=p+1$
    and~$n_q=q+1$, or~$n_p=p-1$ and~$n_q=q-1$. The first possibility
    is not minimal (it is~$\L(p)\cup\L(q)$). The second possibility
    does give minimal coverings; the number is as given since there
    are~$\binom{p+1}{2}$ choices for the index~$p$ lattices
    and~$\binom{q+1}{2}$ choices for the index~$q$ lattices, after
    which the remaining index~$pq$ lattices are determined uniquely.
    
    Finally, (iii) is clear.
\end{proof}

For example, taking~$p=2$ and~$q=3$ or vice versa
in~\Cref{prop:lcm-pq}, we find the following minimal coverings~$\C$
with $\lcm(\C)=6$: there are~$7$ strongly minimal coverings of
size~$6$ by part~(i), and~$18$ minimal, but not strongly minimal,
coverings of size~$7$ by part~(ii) (one of the latter was given
in~\Cref{ex:min-not-strong}); and finally, there are~$12$ minimal but
not strongly minimal coverings of size~$8$ by part~(iii).


\subsection{All minimal coverings of size at most 6}
\Cref{prop:small-n-index-bounds} and~\Cref{prop:lcm-pq}~(i) show that
the nontrivial minimal coverings of size up to~$6$ are as follows:
\begin{enumerate}[(i)]
\item the full coverings~$\L(1)$, $\L(2)$, $\L(3)$, and~$\L(5)$, of
  size~$1$, $3$, $4$, $6$ respectively;
\item a covering obtained from~$\L(2)$ by applying up to three
  $2$-refinements, giving sizes~$4$, $5$, and~$6$;
\item a covering obtained from~$\L(3)$ by applying one $3$-refinement,
  giving size~$6$;
\item one of the seven coverings~$\C$ with~$\lcm(\C)=6$ given
  by~\Cref{prop:lcm-pq}~(i), obtained by~$3$-refining~$\L(2)$
  or~$2$-refining~$\L(3)$.
\end{enumerate}

\begin{table}[ht]
  \caption{All minimal lattice coverings of size up to~$6$}
  \label{tab:size-upto-6}
  \begin{tabular}{|c|l|c|}
    \hline
    Size & index structure & multiplicity \\
    \hline
    1 & (1) & 1 \\
    & total & 1 \\
    \hline
    2 &  -  & 0 \\
    & total & 0 \\
    \hline
    3 & (2,2,2) & 1 \\
    & total & 1 \\
    \hline
    4 & (2,2,(4,4)) &  3 \\
    & (3,3,3,3)   &  1 \\
    & total & 4 \\
    \hline
    5 & (2,(4,4),(4,4)) & 3 \\
    & (2,2,(4,(8,8))) & 6 \\
    & total & 9 \\
    \hline
    6 & ((4,4),(4,4),(4,4))   &  1 \\
    & (2,(4,4),(4,(8,8)))   & 12 \\
    & (2,2,((8,8),(8,8)))   &  3 \\
    & (2,2,(4,(8,(16,16)))) & 12 \\
    & (3,3,3,(9,9,9))       &  4 \\
    & (2,2,(6,6,6,6))       &  3 \\
    & (3,3,3,(6,6,6))       &  4 \\
    & (5,5,5,5,5,5)         &  1 \\
    & total & 40 \\
    \hline
  \end{tabular}
\end{table}

It is straightforward, though tedious, to list by hand all the
coverings obtained by refining the trivial covering.  The result of
carrying this out for sizes up to~$6$ are summarised
in~\Cref{tab:size-upto-6}.  To give the results in a concise form we
introduce some notation:
\begin{itemize}
\item $N$ or~$(N)$ denotes a cocyclic lattice of index~$N$, the
  trivial covering being denoted~$(1)$;
\item $(Np,Np,...,Np)$ with either~$p$ (if~$p\mid N$) or $p+1$
  (if~$p\nmid N$) terms denotes the $p$-refinement of  a cocyclic
  lattice of index~$N$
\end{itemize}
\Cref{tab:size-upto-6} only shows the lattice indices, and only
includes one of each set of similar coverings differing only in the
choice of which lattice of a given index is refined.  Hence, to obtain
literally all the coverings of each size, each entry in the list
should be used multiple times; these multiplicities are shown in the
table.  The indices are bracketed to show the refinement structure.
For example, $(2,2,2)$ denotes the $2$-refinement~$\L(2)$ of the
trivial covering; one of its $2$-refinements, replacing the last
lattice of index~$2$ by its two descendants of relative index~$2$, is
denoted $(2,2,(4,4))$.  This type has multiplicity~$3$, since the
second refinement step could have been applied to any one of the
index~$2$ lattices, so there are in all three coverings of size~$4$ by
lattices of index~$2$, $2$, $4$, and~$4$.

\begin{theorem}
\label{thm:all-upto-6}
  All minimal coverings of size~$n\le6$ are strongly minimal and
  refinements of the trivial covering, and are of one of the types
  shown in the table.  There are a total of~$55$ such coverings,
  including the trivial covering.
\end{theorem}

This result agrees with the lists of coverings
in~\cite{FouvryKoymansIII}. For example, one of the three coverings
included in the line $(2,2,(6,6,6,6))$ of~\Cref{tab:size-upto-6} is
the tenth in the list on~\cite[p.47]{FouvryKoymansIII}, whose lattices
are (in the same order as in~\cite{FouvryKoymansIII} and in our
notation, but writing $(c:d)_N$ for $L(c:d;N)$): $(0:1)_2, (1:0)_6,
(1:1)_2, (1:4)_6, (3:2)_6, (1:2)_6$.  To construct this covering,
start with the trivial covering and apply $2$-refinement to obtain to
the full index-$2$ covering~$\L(2)=\{(0:1)_2, (1:0)_2, (1:1)_2\}$;
then replace~$(1:0)_2$ by its $3$-descendants, which are its
intersection with each of the four index-$3$ lattices
in~$\L(3)=\{(0:1)_3, (1:1)_3, (2:1)_3, (1:0)_3\}$, using the CRT to
write each intersection giving (in the same order) $\{(3:2)_6,
(1:4)_6, (1:2)_6, (1:0)_6\}$. This results in the covering~$\{(0:1)_2,
(3:2)_6, (1:4)_6, (1:2)_6, (1:0)_6, (1:1)_2\}$, which is presented
in~\cite[p.47]{FouvryKoymansIII} by giving a~$\Z$-basis for each of
the~$6$ lattices.

\subsection{All minimal coverings of size 7}
To determine all minimal coverings of size~$7$, we proceed as follows.
First, we list all possible refinements of the trivial cover, starting
with the index lists in~\Cref{tab:size-upto-6}:
\begin{itemize}
\item apply a $2$-refinement to a lattice of even index in a covering of size~$6$;
\item apply a $3$-refinement to a lattice of index not divisible by~$3$ in
  a covering of size~$4$.
\end{itemize}
In principle, we could also have applied a $2$-refinement to a lattice
of odd index in a covering of size~$5$, or a $3$-refinement to a
lattice of index divisible by~$3$ in a covering of size~$5$, but there
are none of these.  This leads to~$11$ different sorted index lists,
which are shown in the first~$11$ lines of ~\Cref{tab:size-7}.
Counting multiplicities, there are~$126$ coverings of size~$7$ which
are refinements of the trivial covering.
\begin{table}[ht]
  \caption{All minimal lattice coverings of size~$7$}
  \label{tab:size-7}
  \begin{tabular}{|l|c|c|}
    \hline
    index structure & multiplicity & strongly minimal? \\
    \hline
    ((4,4),(4,4),(4,(8,8)))     &  6 & yes \\
    (2,(4,(8,8)),(4,(8,8)))     & 12 & yes \\
    (2,((4,4),(8,8),(8,8)))     &  6 & yes \\
    (2,(4,4),(4,(8,(16,16))))   & 24 & yes \\
    (2,2,((8,8),(8,(16,16))))   & 12 & yes \\
    (2,2,(4,((16,16),(16,16)))) &  6 & yes \\
    (2,2,(4,(8,(16,(32,32)))))  & 24 & yes \\
    (2,(4,4),(6,6,6,6))         &  6 & yes \\
    (2,2,(6,6,6,(12,12)))       & 12 & yes \\
    (3,3,3,(6,6,(12,12)))       & 12 & yes \\
    (2,2,(4,(12,12,12,12)))     &  6 & yes \\
    \hline
    total strongly minimal      & 126 & \\
    \hline
    (2,3,3,6,6,6,6)             & 18 & no \\
    \hline
    total not strongly minimal  & 18 & \\
    \hline
    total                       & 144 & \\
    \hline
  \end{tabular}
\end{table}

We also know that there are~$18$ coverings with index
sequence~$(2,3,3,6,6,6,6)$ which are minimal but not strongly minimal,
from~\Cref{prop:lcm-pq}~(ii).

Our theoretical results are not quite strong enough to establish the
fact that there are no additional minimal coverings~$\C$ of
size~$7$.
From~\Cref{thm:Simpson-theorem-2-lattice}, together
with~\Cref{cor:n-p-e-bound-p-power} to eliminate odd prime powers, we
see that~$N=\lcm(\C)$ must be one of the following:
\begin{itemize}
\item $N=2^e$ for $3\le e\le 5$; these~$\C$ are all refinements,
  by~\Cref{thm:all-p-power}, and their index sequences are listed in
  the first seven rows of~\Cref{tab:size-7};
\item $N=6$; these are given by~\Cref{prop:lcm-pq}~(ii), none are
  refinements or strongly minimal, and are shown in the last row
  of~\Cref{tab:size-7};
\item $N=12$; rows 8--11 of~\Cref{tab:size-7} show four different index
  sequences of refinement coverings (which are therefore all strongly
  minimal) in this case.
\end{itemize}
To establish that there are no more minimal
coverings~$\C$ of size~$7$ with~$\lcm(\C)=12$, we rely on machine
computations.  Running our code from~\cite{CK} with~$n=7$, which takes
under~$2$ minutes of computer time, yields precisely the~$144$ minimal
coverings we have already found, including $126$ strongly minimal
refinements which are all refinements coverings, and the
additional~$18$ which are not strongly minimal.

\begin{theorem}
\label{thm:all-7}
  All minimal covering of size~$n=7$ are either strongly minimal
  coverings, of which there are~$126$, all refinements of the trivial
  covering, or are one of~$18$ additional coverings, which are not
  strongly minimal, making~$144$ minimal coverings of size~$7$ in
  all. The structure of these is given in~\Cref{tab:size-7}.
\end{theorem}

\subsection{All minimal coverings of size 8}
For size~$8$ we proceeded as for size~$7$, finding~$550$ strongly
minimal coverings, all refinements of the trivial covering, with~$32$
distinct sorted index sequences in~$35$ different refinement types.

In addition, there are~$174$ minimal coverings of size~$8$ which are
not strongly minimal:
\begin{itemize}
\item $12$ with index sequence~$(2,3,6,6,6,6,6,6)$, of the type
  described in~\Cref{prop:lcm-pq}~(iii);
\item $18$ with index sequence~$(3,3,4,4,6,6,6,6)$; these are
  $2$-refinements of the size~$7$ coverings with index
  sequence~$(2,3,3,6,6,6,6)$, refining the lattice of index~$2$;
\item $72$ with index sequence~$(2,3,3,6,6,6,12,12)$; these are also
  $2$-refinements of the size~$7$ coverings with index
  sequence~$(2,3,3,6,6,6,6)$, refining one of the four lattices of
  index~$6$;
  \item $72$ with index sequence~$(2,3,3,4,6,6,12,12)$. These consist
    of: any one lattice of index~$2$; any one of index~$4$ not
    contained in it; any two lattices of index~$3$; and two lattices
    each of index~$6$ and~$12$ separated from each other and the
    preceding ones.
\end{itemize}
One example of the last type of covering is
\begin{align*}
  & L(1:1;2),  L(1:2;4),  L(1:0;3), L(1:1;3),\\
  & L(0:1;6), L(2:1;6),  L(1:8;12), L(3:4;12).
\end{align*}
Running our program~\cite{CK} with~$n=8$ takes just under~$4$ hours of computer
time.

\begin{theorem}
\label{thm:all-8}
  All minimal covering of size~$n=8$ are either strongly minimal
  coverings, of which there are~$550$, all refinements of the trivial
  covering, or are one of~$174$ additional coverings of four types,
  which are not strongly minimal, making~$724$ minimal coverings of
  size~$8$ in all. The structure of these is given
  in~\Cref{tab:size-8} in Appendix~A.
\end{theorem}

\section{Covering systems of congruences}
\label{sec:congruences}
We briefly note here the analogues of the results proved above for
lattice coverings, and which therefore apply to projective covering
systems, to the context of covering systems, where most of them are
trivial.  Recall that we denote the residue class~$a+N\Z$ by~$R(a;N)$.
\begin{enumerate}[(i)]
\item If~$N=N_1N_2$ with~$N_1,N_2$ coprime, then there is a bijection
  between pairs of classes~$R(a_1;N_1), R(a_2;N_2)$ and single
  classes~$R(a;N)$.  (This is the usual Chinese Remainder Theorem.)
\item For a fixed positive integer~$N$, there is a bijection between
  the distinct residue classes~$R(a;N)$ and~$\Z/N\Z$ (that is, the
  affine line over~$\Z/N\Z$).
\item The number of residue classes modulo~$N$ is~$N$.
\item For a fixed positive integer~$N$, every integer~$a$ belongs to
  exactly one residue class modulo~$N$, namely~$R(a;N)$.
  \item If $M\mid N$ then every class~$R(a;M)$ modulo~$M$ is the
    disjoint union of~$N/M$ classes modulo~$N$, namely $R(a+kM;N)$
    for $0\le k < N/M$.
\item A \emph{covering system} is a finite collection of
  classes~$R(a_i;N_i)$ whose union is~$\Z$.  It is \emph{minimal} if
  none of the classes can be replaced by a strict subclass without
  losing the covering property, and \emph{strongly minimal} if every
  integer belongs to exactly one class in the covering.  These have been studied in the literature under various names such as ``exactly covering system'' or ``disjoint covering system'', see for example~\cite{NZ, Znam, Simpson2}.
\item The trivial covering system is the one with just one
  class, $R(0;1)$.
\item The set of all classes~$R(a;N)$ modulo~$N$ (for $0\le a<N$) is a
  strongly minimal covering, the \emph{full covering modulo~$N$}.
\item Not every minimal covering is strongly minimal.  For example:
  \[
  R(0;2) \cup R(0;3) \cup R(1;6) \cup R(5;6)
  \]
  is a minimal covering, but integers divisible by~$6$ are contained in
  both the first two classes.
\item If $\{R(a_i;N_i)\}$ is a covering system, then
  $\sum_i1/N_i\ge1$, with equality if and only if the covering is
  strongly minimal.
\item Given a covering system~$\{R(a_i;N_i)\}$ and a prime~$p$, a
  $p$-refinement of the covering system is any new covering obtained
  by replacing one of the~$R(a_i;N_i)$ by its~$p$ subclasses of
  modulus~$pN_i$.  (The exact same construction works when~$p$ is an
  arbitrary positive integer.)  Repeating the process any finite
  number of times, we obtain a general refinement of the original
  covering.  If~$N$ is the least common multiple of the original
  moduli~$N_i$, then by repeated $p$-refinement we can reach a
  covering in which all the moduli are equal to~$N$, possibly with
  repeats. This will have no repeats if and only if the
  original system was strongly minimal.
\item Every refinement of the trivial covering is strongly minimal,
  and has the property that the moduli are not coprime (being all
  divisible by the prime used in the first refinement step).
\item Not every strongly minimal covering system can be obtained by
  refining the trivial covering:  for example,
  $R(0;6)$, $R(1,10)$, $R(2,15)$ are pairwise disjoint and covering
  $5+3+2=11$ of the residue classes modulo~$30$; if we include the
  remaining~$19$ classes~$R(a;30)$ we obtain a strongly minimal
  covering, where the moduli are jointly coprime ($6$, $10$, $15$, and~$30$).
\item (See~\cite[Cor.~2]{Simpson}.) A minimal covering system by~$n$
  classes whose moduli have least common multiple with
  factorization~$\prod_{i=1}^{t}p_i^{e_i}$ satisfies
  \[
  n \ge 1 + \sum_{i=1}^{t}e_i(p_i-1).
  \]
\end{enumerate}

The literature about covering systems concentrates on  the study of
systems where no two classes have the same moduli.  An example due to
Erd\H{o}s is
\[
R(0;2) \cup R(1;4) \cup R(3;8) \cup R(0;3) \cup R(7;12) \cup R(23;24),
\]
which is minimal but not strongly minimal: $\sum_i1/N_i=32/24 =
4/3$. The analogous problem for lattice coverings would be to ask for
coverings by lattices with distinct indices; we have not studied this.

\addresshere

\newpage
\appendix
\section{Coverings of size 8}
\label{aTable}
\begin{table}[ht]
  \caption{Strongly minimal lattice coverings of size~$8$}
  \label{tab:size-8}
  \begin{tabular}{|l|c|c|}
    \hline
    index structure & multiplicity & strongly minimal? \\
    \hline
%
  (2,2,(4,(8,(16,(32,(64,64))))))  & 48 & yes \\
  (2,2,(4,(8,((32,32),(32,32)))))  & 12 & yes \\
  (2,2,(4,((16,16),(16,(32,32))))) & 24 & yes \\
  (2,2,((8,8),(8,(16,(32,32)))))   & 24 & yes \\
  (2,2,((8,8),((16,16),(16,16))))  &  6 & yes \\
  (2,2,(8,(16,16)),(8,(16,16))))   & 12 & yes \\
  (2,(4,4),(4,(8,(16,(32,32)))))   & 48 & yes \\
  (2,(4,4),(4,((16,16),(16,16))))  & 12 & yes \\
  (2,(4,4),((8,8),(8,(16,16))))    & 24 & yes \\
  (2,(4,(8,8)),(4,(8,(16,16))))    & 48 & yes \\
  (2,(4,(8,8)),((8,8),(8,8)))      & 12 & yes \\
  ((4,4),(4,4),(4,(8,(16,16))))    & 12 & yes \\
  ((4,4),(4,4),((8,8),(8,8)))      &  3 & yes \\
  ((4,4),(4,(8,8)),(4,(8,8)))      & 12 & yes \\ 
    %
    %
    (3,3,3,(9,9,(27,27,27)))       & 12 & yes \\
    (3,3,(9,9,9),(9,9,9))          &  6 & yes \\ 
%
%
  (2,2,(4,(8,(24,24,24,24))))      & 12 & yes \\
  (2,2,(4,(12,12,12,(24,24))))     & 24 & yes \\
  (2,2,(6,6,6,(12,(24,24))))       & 24 & yes \\
  (2,2,(6,6,6,(18,18,18)))         & 12 & yes \\
  (2,2,(6,6,(12,12),(12,12)))      & 18 & yes \\
  (2,2,((8,8),(12,12,12,12)))      &  6 & yes \\
  (2,(4,4),(4,(12,12,12,12)))      & 12 & yes \\
  (2,(4,4),(6,6,6,(12,12)))        & 24 & yes \\
  (2,(4,(8,8)),(6,6,6,6))          & 12 & yes \\
  (3,3,3,(6,6,(12,(24,24))))       & 24 & yes \\
  (3,3,3,(6,(12,12),(12,12)))      & 12 & yes \\
  (3,3,3,(6,6,(18,18,18)))         & 12 & yes \\
  (3,3,3,(9,9,(18,18,18)))         & 12 & yes \\
  (3,3,(6,6,6),(6,6,6))            &  6 & yes \\
  (3,3,(6,6,6),(9,9,9))            & 12 & yes \\
  ((4,4),(4,4),(6,6,6,6))          &  3 & yes \\ 
  %
  %
  (2,2,(10,10,10,10,10,10))        & 3 & yes \\
  (5,5,5,5,5,(10,10,10))           & 6 & yes \\
  %
  %
  (7,7,7,7,7,7,7,7)                & 1 & yes \\ 
  \hline
  total strongly minimal           & 550 & \\
  \hline
  (2,3,6,6,6,6,6,6)                & 12 & no \\
  (3,3,4,4,6,6,6,6)                & 18 & no \\
  (2,3,3,6,6,6,12,12)              & 72 & no \\
  (2,3,3,4,6,6,12,12)              & 72 & no \\
    \hline
  total not strongly minimal       & 174 & \\
  \hline
  total                            & 724 & \\ 
  \hline 
  \end{tabular}
\end{table}
\end{document}